\documentclass{amsart}

\usepackage{xpatch,amssymb,amsfonts}
\usepackage{hyperref}
\usepackage[alphabetic]{amsrefs}
\usepackage{breakurl}
\usepackage{mathtools}
\usepackage[all,arc]{xy}
\usepackage{enumerate}
\usepackage[inline]{enumitem}
\usepackage{mathrsfs}
\usepackage{nicefrac}
\usepackage{tikz}
\usepackage{tikz-cd}
\usepackage{array}
\usepackage{esint}
\usepackage{graphicx}
\usepackage[LGR,T1]{fontenc}
\usepackage[utf8]{inputenc}
\usepackage{enumitem}
\usepackage{float}

\usepackage{thm-restate}

\newtheorem{thm}{Theorem}[section]
\newtheorem*{thm*}{Theorem}

\newtheorem{prop}[thm]{Proposition}

\newtheorem{lem}[thm]{Lemma}
\newtheorem{conj}[thm]{Conjecture}

\theoremstyle{definition}
\newtheorem{defn}[thm]{Definition}

\newtheorem{exmp}[thm]{Example}

\newtheorem{rems}[thm]{Remarks}

\numberwithin{equation}{section}

\DeclareRobustCommand{\SkipTocEntry}[5]{} 

\makeatletter
\def\@seccntformat#1{%
  \protect\textup{\protect\@secnumfont
    \ifnum\pdfstrcmp{subsection}{#1}=0 \bfseries\fi
    \csname the#1\endcsname
    \protect\@secnumpunct
  }%
}  
\makeatother

\makeatletter
\@namedef{subjclassname@2020}{\textup{2020} Mathematics Subject Classification}
\makeatother

\renewcommand{\phi}{\varphi}
\renewcommand{\subset}{\subseteq}
\renewcommand{\supset}{\supseteq}

\DeclareMathOperator{\SL}{{\rm SL}}

\DeclareMathOperator{\id}{{\rm id}}
\let\Im\relax

\DeclareMathOperator{\Im}{{\rm Im}}

\DeclareMathOperator{\PSL}{{\rm PSL}}
\DeclareMathOperator{\Aut}{{\rm Aut}}
\DeclareMathOperator{\Mod}{{\rm Mod}}
\DeclareMathOperator{\PMod}{{\rm PMod}}
\DeclareMathOperator{\Teich}{{\rm Teich}}
\DeclareMathOperator{\Homeo}{{\rm Homeo}}

\DeclareMathOperator{\Deck}{{\rm Deck}}

\newcommand{\bbD}{\mathbb{D}}
\newcommand{\bbR}{\mathbb{R}}
\newcommand{\bbQ}{\mathbb{Q}}
\newcommand{\bbZ}{\mathbb{Z}}
\newcommand{\bbN}{\mathbb{N}}

\newcommand{\bbC}{\mathbb{C}}
\newcommand{\bbH}{\mathbb{H}}

\newcommand{\hQ}{\widehat{\mathbb{Q}}}
\newcommand{\hR}{\widehat{\mathbb{R}}}
\newcommand{\hC}{\widehat{\mathbb{C}}}

\newcommand{\tphi}{\widetilde{\phi}}
\newcommand{\tpsi}{\widetilde{\psi}}
\newcommand{\tgamma}{\widetilde{\gamma}}

\newcommand{\talpha}{\widetilde{\alpha}}
\newcommand{\tbeta}{\widetilde{\beta}}
\newcommand{\tH}{\widetilde{H}}
\newcommand{\tK}{\widetilde{K}}
\newcommand{\tL}{\widetilde{L}}
\newcommand{\tl}{\tilde{l}}
\newcommand{\trho}{{\widetilde{\rho}}}

\newcommand{\calA}{\mathcal{A}}
\newcommand{\calC}{\mathcal{C}}

\newcommand{\calJ}{\mathcal{J}}
\newcommand{\calF}{\mathcal{F}}
\newcommand{\calT}{\mathcal{T}}

\newcommand{\calM}{\mathcal{M}}
\newcommand{\calW}{\mathcal{W}}

\newcommand{\rel}{\,\,{\rm rel.}\,}
\newcommand{\from}{\colon}
\newcommand{\dto}{\rightrightarrows}

\renewcommand{\emptyset}{\varnothing}

\DeclarePairedDelimiter{\norm}{\lVert}{\rVert} 

\title{Curve Attractors for Marked Rational Maps}
\author{Zachary Smith}
\address{Department of Mathematics, University of California, Los Angeles, CA 90095}
\email{zbsmith@math.ucla.edu}
\date{\today}
\thanks{The author was partially supported by NSF grant DMS-2054987.}

\DefineSimpleKey{bib}{primaryclass}{}
\DefineSimpleKey{bib}{archiveprefix}{}
\BibSpec{arXiv}{%
  +{}{\PrintAuthors}{author}
  +{,}{ \textit}{title}
  +{}{ \parenthesize}{date}
  +{,}{ arXiv }{eprint}
  +{,}{ primary class }{primaryclass}
}

\AtBeginDocument{%
   \def\MR#1{}
}

\begin{document}

\maketitle

\begin{abstract}
A Thurston map $f\from (S^2, A) \to (S^2, A)$ with marking set $A$ induces a pullback relation on isotopy classes of Jordan curves in $(S^2, A)$. If every curve lands in a finite list of possible curve classes after iterating this pullback relation, then the pair $(f,A)$ is said to have a {\em finite global curve attractor}. It is conjectured by Pilgrim \cite{pilgrimsurvey} that all rational Thurston maps that are not flexible Latt\`{e}s maps have a finite global curve attractor. We present partial progress on this problem. Specifically, we prove that if $A$ has four points and the postcritical set (which is a subset of $A$) has two or three points, then $(f,A)$ has a finite global curve attractor.

We also discuss extensions of the main result to certain special cases where $f$ has four postcritical points and $A=P_f$. Additionally, we speculate on how some of these ideas might be used in the more general case.

\end{abstract}

\tableofcontents

\section{Introduction}

This paper is concerned with dynamical questions which arise in the context of Thurston maps. A {\em Thurston map} is an orientation-preserving branched cover $f\from S^2\to S^2$ that is not a homeomorphism and such that each of its critical points has finite forward orbit; it is considered together with the data of a finite set of marked points $A$ that contains said forward orbits. These maps were introduced by William Thurston as combinatorial models of postcritically-finite rational maps. In 1982 Thurston presented his celebrated characterization theorem, which gives necessary and sufficient conditions for when a Thurston map is “realized” by a rational map in a suitable sense. The proof, as explicated by Douady and Hubbard \cite{dhthurston}, uses an analytic map on Teichm\"{u}ller space $\sigma_f\from\calT_A\to\calT_A$ which is induced by $f$ by pulling back complex structures. The question of whether $f$ is realized by a rational map then reduces to whether $\sigma_f$ has a fixed point in $\calT_A$.

Interestingly, the main requirement in Thurston's theorem is the (non)existence of Jordan curves in $S^2\setminus A$ with certain invariance properties. More generally, every Thurston map induces a pullback relation on isotopy classes of Jordan curves in $S^2\setminus A$. The restriction $f:S^2\setminus f^{-1}(A)\to S^2\setminus A$ is a covering map, so if $\gamma\subset S^2\setminus A$ is a Jordan curve, then a component $\widetilde{\gamma}$ of $f^{-1}(\gamma)$ will also be a Jordan curve in $S^2\setminus A$. We say that $\widetilde{\gamma}$ is a {\em pullback} of $\gamma$ by $f$. Lifting isotopies shows that the set of isotopy classes of $f^{-1}(\gamma)\rel A$ depends only on the isotopy class of $\gamma\rel A$.

With this in mind, understanding the dynamics of Thurston maps $f$, finding suitable invariants by which to classify such maps, and understanding the dynamics of the associated pullback map $\sigma_f$ are all closely related to the curve pullback relation described above. See, for example, \cite{bekp},\cite{selthesis}, \cite{SarahCritEndo}, \cite{KPSlifts}, and \cite{pilgrimsurvey}.

One of the major open problems in the study of the pullback relation on curves is the following (see \cite{lodgethesis} and \cite{pilgrimsurvey}):

\begin{conj}[Global Curve Attractor Conjecture]
Let $f\from S^2\to S^2$ be a Thurston map with hyperbolic orbifold that is realized by a rational map. Then there is a finite set $\calA(f)$ of Jordan curves in $S^2\setminus A$ with the following property: for every Jordan curve $\gamma$ in $S^2\setminus A$ there is a positive integer $N(\gamma)$ such that, for $n\geq N(\gamma)$, all pullbacks $\widetilde{\gamma}$ of $\gamma$ under $f^n$ are contained in $\calA(f)$ up to isotopy $\rel A$.
\end{conj}

Stated less formally, if one iterates the pullback relation on curves, then one eventually lands in a finite set of isotopy classes. This problem appears to be quite difficult, as even in the simplest case of four marked points (so $|A|=4$) there are relatively few classes of rational maps for which it is known.

\subsection{Main Results}

In this paper we will verify the curve attractor conjecture in the following special case:

\begin{restatable}{thm}{resmainthm}
\label{mainthm}
Let $f\from\hC\to\hC$ be a rational Thurston map with a set $A$ of four marked points. If the postcritical set $P_f\subset A$ has at most three points, then $(f,A)$ has a finite global curve attractor.
\end{restatable}

There are two main ingredients to the proof. The first is the observation that the pullback map on Teichm\"{u}ller space $\sigma_{f,A}$ covers a {\em map} on moduli space, and this map will essentially be the same as $f$ itself. This allows one to explicitly relate the dynamics of $\sigma_f$ to the dynamics of $f$. The correspondence we will use has been studied before by Kelsey and Lodge in \cite{kelseylodgequadratic}, where they prove all rational quadratic Thurston maps with at most four postcritical points have a finite global curve attractor. Their technique was to enumerate all possible combinatorial classes of such maps and then calculate the virtual endomorphism for each. We will take a very different tack in our argument, which will circumvent the need for the virtual endomorphism entirely and hence scale better with degree.

This brings us to the second main ingredient of our proof, which is the ``leashing'' argument. For four marked points the pullback relation on curves is encoded in the boundary of behavior of $\sigma_f$ at rational cusps $\hQ = \bbQ\cup\{\infty\}$, so it suffices to analyze the dynamics of these cusps under $\sigma_f$. We will endow a certain $\sigma_f$-invariant subset $X$ of of Teichm\"{u}ller space with a path-metric obtained by pulling back an orbifold metric on $\hC\setminus P_\calF\supset \hC\setminus P_f$, and $\sigma_f$ will be uniformly contracting with respect to this metric. The subset $X$ will be a {\em truncated} hyperbolic space, meaning that it is the complement of a disjoint collection of horoballs. These horoballs will all be based at rational cusps of Teichm\"{u}ller space identified with $\bbH$. From there, we attach a ``leash'' from the fixed point in Teichm\"{u}ller space to these complementary horoballs, and using that $\sigma_f$ is contractive on $X$, ``tighten'' this leash by iterating $\sigma_f$. There will only be finitely many horoballs in the family that we may land on after this procedure, and the base cusps of these horoballs will exactly be the desired attractor on cusps.

The main theorem also extends to rational Thurston maps with four postcritical points that are obtained by ``twisting'' a rational map with fewer postcritical points by a M\"{o}bius transformation. As shown in \cite{kelseylodgequadratic}, such maps are characterized by a certain condition on their static portrait. In particular, we will prove

\begin{restatable}{thm}{resmainthmtwo}
\label{mainthm2}
Let $f\from\hC\to\hC$ be a rational Thurston map with four postcritical points. If one of the postcritical points of $f$ is statically trivial, then $(f,P_f)$ has a finite global curve attractor.
\end{restatable}

We delay the precise definition of static triviality until Section \ref{extensionsection}.

Finally, we remark that the extremal cases of Theorem \ref{mainthm} (i.e., when $f$ is either hyperbolic or expanding) can be derived as a corollary to a result concerning the general multivalued correspondence on moduli space (see \cite[Theorem 7.2]{KPSlifts}). This result is not able to handle cases where there are a mixture of attractive and repelling dynamics on moduli space though, which is one of the major features of our theorem. It is possible that the construction of a suitable analog of the orbifold metric and leashing trick used in our proof can be implemented in the general case. We discuss this idea further at the end of the paper.

\subsection{Outline of Paper}

The paper is organized as follows. In Section \ref{preliminarysection} we review the essential definitions and background material.

In Section \ref{modmapsection} we establish the fundamental commutative diagram; a map $(f,A)$ with four marked points $A = \{0,1,\infty,a\}$ and three postcritical points $P_f = \{0,1,\infty\}$ either has constant Thurston pullback map, or else covers a map on moduli space. We give a reasonably thorough proof of this correspondence.

In the first half of Section \ref{gcasection} we introduce some additional tools. Specifically, we introduce a preferred family of horoballs based at rational cusps and describe the images of these horoballs under deck transformations of the universal covering map $\pi\from\bbH\to\hC\setminus\{0,1,\infty\}$, as well as their images under the Thurston pullback map. We divide rationals cusps into two groups---roughly speaking, those with attracting dynamics and those with repelling dynamics. We also design a metric (induced by a suitably chosen orbifold metric) with respect to which $\sigma_f$ will satisfy nice contractive properties.

In the second half of Section \ref{gcasection} we prove the main theorem. The proof largely consists of combining the tools above into the previously described ``leashing'' argument. We also provide several examples of marked rational Thurston maps with nontrivial pullbacks that satisfy the hypotheses of Theorem \ref{mainthm}, and thus have finite global curve attractors.

In Section \ref{extensionsection}, we describe how to produce further examples of rational Thurston maps with four postcritical points by composing maps satisfying Theorem \ref{mainthm} with either M\"{o}bius transformations or flexible Latt\`{e}s maps, and give examples for both cases. Examples obtained from the former construction can be detected by a simple condition on the portrait of the Thurston map (see \cite[Proposition 2.5]{kelseylodgequadratic}), which proves Theorem \ref{mainthm2}.

Finally, in Section \ref{futuresection} we describe how some of the ideas of this paper might be used to approach the still-open general case of the curve attractor problem. We also pose some questions.

\subsection*{Acknowledgments}\addtocontents{toc}{\SkipTocEntry} 
I would like to thank Mario Bonk, my PhD advisor, for both suggesting this topic and for his constant support and guidance. I would also like to thank Mikhail Hlushchanka and Kevin Pilgrim for various helpful comments and remarks.

\section{Preliminaries}\label{preliminarysection}

\subsection{Notations and Conventions}

Let $S^2$ denote a topological $2$-sphere, which we will always take to be oriented. When we wish to work with explicit coordinates, we will take the Riemann sphere as the underlying conformal structure and write $S^2=\hC$.

The notation $\bbH$ will be used for the hyperbolic plane $\bbH = \{\tau\in\bbC: \Im \tau>0\}$.

For a given set $A$, we denote the cardinality of $A$ by $|A|$.

Let $X$ and $Y$ be topological spaces. A continuous map $H\from X\times[0,1]\to Y$ is called an {\em isotopy} if $H_t:=H(\cdot,t)\from X\to Y$ is a homeomorphism for each $t\in [0,1]$. The map $H_t$ is called the {\em time-t} map of the isotopy. Given a distinguished subset $A\subset X$, the map $H$ is said to be an {\em isotopy relative to $A$} if $H_t|_A = H_0|_A$ for each $t\in [0,1]$. Two homeomorphisms $h_0,h_1\from X\to Y$ are called {\em isotopic} ({\em relative to $A$}) if there exists an isotopy (relative to $A$) $H\from X\times [0,1]\to Y$ with $H_0 = h_0$ and $H_1=h_1$. If $h_0$ and $h_1$ are isotopic relative to $A$, we will write $h_0\simeq h_1\rel A$.

\subsection{Basic Definitions}

In this section we give a summary of some basic aspects of Thurston theory. We omit proofs of the stated propositions, but they can be found in any standard reference on the subject (e.g., \cite{dhthurston} or \cite{markedthurston}).

Let $f\colon S^2\to S^2$ be an orientation-preserving branched covering map. We denote its degree by $\deg(f)$. A point $c\in S^2$ where the local mapping degree $\deg(f,c)$ is at least $2$ is a {\em critical point}, and we denote the set of critical points by $C_f$. A {\em postcritical point} of $f$ is a point $p\in S^2$ of the form $p = f^n(c)$ where $c$ is a critical point of $f$ and $f^n$ denotes the $n$th iterate of $f$ with $n$ a nonnegative integer. We denote the set of all postcritical points $P_f$, so
\[P_f = \bigcup_{n\geq 1}\{f^n(c): c\in C_f\}.\]
If $|P_f|$ is finite, then $f$ is said to be {\em postcritically-finite}.

\begin{defn} Let $f\from S^2\to S^2$ be an orientation-preserving postcritically-finite branched covering map with $\deg(f)\geq 2$. Let $A\subset S^2$ be a finite set of marked points with the properties $P_f\subset A$ and $f(A)\subset A$. We call the map of pairs $f\from (S^2,A)\to (S^2,A)$ a {\em Thurston map}.\end{defn}

\begin{rems}
\,
\begin{enumerate}[leftmargin=2em]
\item In the above definition it may be the case that $P_f$ is a proper subset of $A$. This slightly extends the standard definition of Thurston maps, where $A=P_f$. We will sometimes use the phrase {\em marked Thurston map} to emphasize the more general setting. We refer the reader to \cite{markedthurston} for a general overview of Thurston theory with markings and how it differs from the original setting of \cite{dhthurston}.
\item The $(\cdot,A)$ portion of the notation will be suppressed when the context is clear.
\item Conversely, when we wish to emphasize that a map $f$ is being considered relative to a particular set of marked points, we may write $(f,A)$.
\item We will assume throughout our discussions that $|A|\geq 3$.
\end{enumerate}
\end{rems}

\begin{defn}
For a Thurston map $(f,A)$, the {\em dynamical portrait} of $f$ is a directed graph with vertex set $C_f\cup A$, where a vertex $z$ is connected to a vertex $w$ by a directed edge when $f(z)=w$. The directed edges carry weighting $\deg(f,z)$.

The {\em static portrait}, meanwhile, is a bipartite directed graph where the vertex set is the disjoint union of $C_f\cup A$ and $A$, where directed edges originate from vertices in the first set and terminate in the latter, and we again weight an edge connecting $z\in C_f\cup A$ to $w=f(z)\in A$ by $\deg(f,z)$.
\end{defn}

\begin{defn}
Two marked Thurston maps $(f,A)$ and $(g,A')$ are said to be {\em Thurston equivalent} or {\em combinatorially equivalent} if there exist orientation-preserving homeomorphisms $h_0,h_1\from (S^2,A)\to (S^2,A')$ such that the diagram

\[\begin{tikzcd}
	{(S^2,A)} && {(S^2,A')} \\
	\\
	{(S^2,A)} && {(S^2,A')}
	\arrow["f"', from=1-1, to=3-1]
	\arrow["{h_1}", from=1-1, to=1-3]
	\arrow["{h_0}"', from=3-1, to=3-3]
	\arrow["g", from=1-3, to=3-3]
\end{tikzcd}\]
commutes, and also $h_0\simeq h_1 \rel A$.
\end{defn}

\begin{defn}\label{teichdef}
The {\em Teichm\"{u}ller space} of a marked sphere $(S^2,A)$, denoted $\calT_A$ or $\Teich(S^2,A)$, is the space of equivalence classes of orientation-preserving homeomorphisms $\phi\from S^2\to \hC$ where two such homeomorphisms have $\phi_0\sim\phi_1$ if and only if there is a M\"{o}bius transformation $M\from\hC\to\hC$ such that $\phi_1\simeq M\circ \phi_0 \rel A$ and also the diagram

\[\begin{tikzcd}
	& \hC \\
	{(S^2,A)} \\
	& \hC
	\arrow["{\phi_0}", from=2-1, to=1-2]
	\arrow["M", from=1-2, to=3-2]
	\arrow["{\phi_1}"', from=2-1, to=3-2]
\end{tikzcd}\]
commutes on $A$. The maps $\phi$ are sometimes called {\em markings} or {\em marking homeomorphisms}.
\end{defn}

\begin{defn}\label{modulispacedef}
The {\em moduli space} of a marked sphere $(S^2,A)$, denoted $\calM_A$, is the space of equivalence classes of injections $\iota\from A\hookrightarrow \hC$ where $\iota_0\sim \iota_1$ if and only if there is a M\"{o}bius transformation $M\from\hC\to\hC$ such that $\iota_1= M\circ \iota_0$.
\end{defn}

There is an natural projection map $\pi\from \calT_A\to \calM_A$  defined by $\pi([\phi]) = [\phi|_A]$. This map is holomorphic with respect to the complex structure on $\calT_A$ and is also a universal covering map.

To each Thurston map $f\from (S^2,A)\to (S^2,A)$ there is an associated pullback map on Teichm\"{u}ller $\sigma_{f,A}\from \calT_A\to\calT_A$. To construct this map, we first note the following proposition:

\begin{prop}\label{pullbackdefdiagram}
Let $(f,A)$ be a Thurston map, and let $\phi\from (S^2, A)\to(\hC,\phi(A))$ be an orientation-preserving homeomorphism. There is an orientation-preserving homeomorphism $\psi\from (S^2,A)\to (\hC,\psi(A))$ and a rational map $F\from (\hC,\psi(A))\to (\hC,\phi(A))$ so that the following diagram commutes:

\[\begin{tikzcd}
	{(S^2,A)} && {(\hC,\psi(A))} \\
	&&&\\
	{(S^2,A)} && {(\hC,\phi(A))}.
	\arrow["\psi", from=1-1, to=1-3]
	\arrow["\phi"', from=3-1, to=3-3]
	\arrow["f"', from=1-1, to=3-1]
	\arrow["{F}", from=1-3, to=3-3]
\end{tikzcd}\]
Moreover, the map $\psi$ is unique up to equivalence with respect to the relation $\sim$ defining $\calT_A$.
\end{prop}

\begin{defn}
Given a Thurston map $f\from (S^2,A)\to (S^2,A)$ the associated {\em Thurston pullback map} $\sigma_{f,A}\from\calT_A\to\calT_A$ is the map $[\phi]\mapsto [\psi]$, where $\psi$ is the map provided by the previous proposition.
\end{defn}

\begin{rems}
\,
\begin{enumerate}[leftmargin=2em]
\item The ``$\cdot,A$'' portion of the subscript is to emphasize the more general setting of a marked Thurston map $(f,A)$, since $\sigma_f$ typically denotes the pullback of $(f,P_f)$. Since we work almost exclusively in the marked setting in this paper, we will suppress the subscript when the context is clear and simply write $\sigma_f:= \sigma_{f,A}$.
\item We also remark that $\sigma_f$ can be defined even when the requirement $\deg(f)\geq 2$ is dropped from the definition of a Thurston map, i.e., when $f$ is simply a homeomorphism for which $f(A)=A$.
\end{enumerate}
\end{rems}

We mention one further generalization of Thurston maps as it will be useful in developing the basic composition laws for pullback maps on Teichm\"{u}ller space.

\begin{defn}\label{admissiblecover}
We call an orientation-preserving branched cover $f\from (S^2,A)\to (S^2,B)$ for $2<|A|,|B|<\infty$ {\em admissible} if: (i) $f(C_f)\subset B$, i.e., $B$ contains the critical values of $f$; and (ii) $f(A)\subset B$.
\end{defn}

This nondynamical setting has been considered by Koch (see, for instance, \cite{SarahCritEndo} and \cite{bekp}). There is a well-defined holomorphic pullback map $\sigma_f\from \calT_B\to\calT_A$ in this case as well, and it is defined analogously to the dynamical setting.

We are now ready to state the composition rule for pullback maps:

\begin{prop}\label{pullbackcomp}
Suppose $f\from (S^2,A)\to(S^2,B)$ and $g\from (S^2, B)\to (S^2,C)$ are two admissible branched covers. Then $g\circ f\from (S^2,A)\to (S^2,C)$ is an admissible branched cover, and $\sigma_{g\circ f} = \sigma_f\circ \sigma_g$. In other words, the following diagram commutes:
\[
\begin{tikzcd}
	{\calT_C} && {\calT_B} && {\calT_A}.
	\arrow["{\sigma_g}", from=1-1, to=1-3]
	\arrow["{\sigma_f}", from=1-3, to=1-5]
	\arrow["{\sigma_{g\circ f}}"',bend right, from=1-1, to=1-5]
\end{tikzcd}
\]
\end{prop}

\subsection{Pullback Relation on Curves}

Consider a Thurston map $f\from S^2\to S^2$ with marked set $A$. We will call a curve $\gamma\subset S^2$ a {\em Jordan curve in the marked sphere} $(S^2,A)$ if $\gamma\subset S^2\setminus A$ and $\gamma$ is simple and closed. As stated in the introduction, there is a pullback relation on Jordan curves in $(S^2,A)$ defined in the following manner: if $\gamma$ is a Jordan curve in $S^2\setminus A$, then a connected component $\tgamma$ of $f^{-1}(\gamma)$ is a {\em pullback} of $\gamma$ by $f$. The restriction of $f\from S^2\setminus f^{-1}(A)\to S^2\setminus A$ is an ordinary covering map, so each pullback $\tgamma$ of $\gamma$ is itself a Jordan curve in $(S^2,A)$. The restriction $f|_{\tgamma}: \tgamma\to \gamma$ is also a covering map, and so has some finite degree, which we denote $\deg(f\from\tgamma\to\gamma)$.

A Jordan curve $\gamma\subset S^2\setminus A$ is {\em essential} if each of the two connected components of $S^2\setminus \gamma$ contain at least two points of $A$. We call a Jordan curve {\em peripheral} if it is not essential (so $\gamma$ either encircles a single marked point or is nullhomotopic in $S^2\setminus A$).

By lifting isotopies one can see that the isotopy classes of curves in $f^{-1}(\gamma)\rel A$ depends only on the isotopy class of $\gamma\rel A$, and not on the specific choice of $\gamma$. More precisely, we have the following proposition:

\begin{prop}\label{curverelA}
Let $(f,A)$ be a Thurston map and let $\alpha$ and $\beta$ be Jordan curves in $(S^2,A)$ with $\alpha\simeq\beta\rel A$. Then there is a bijection $\talpha \leftrightarrow \tbeta$ between the pullbacks $\talpha$ of $\alpha$ and the pullbacks $\tbeta$ of $\beta$ under $f$ such that, for all pullbacks corresponding under this bijection, we have $\talpha\simeq \tbeta\rel A$ and $\deg(f\from\talpha\to \alpha) = \deg(f\from\tbeta\to \beta)$.
\end{prop}

See \cite[Lemma 6.9]{bmexpanding} for a proof.

Thus there is a well-defined {\em pullback relation} on the set of isotopy classes of Jordan curves in $S^2\setminus A$. In general, given a Jordan curve $\gamma$ in $(S^2,A)$ the set of pullbacks $f^{-1}(\gamma)$ may have curves belonging to distinct isotopy classes. A special property of the case $|A|=4$ is that the essential pullbacks (if there are any) of an essential Jordan curve belong to only one isotopy class. As the case of four marked points is our primary interest we will maintain this setting for the rest of this section.

\begin{defn}
For a Thurston map $(f,A)$ with $|A|=4$, let $\gamma$ be an essential Jordan curve of $(S^1,A)$. Let $\gamma_1,\dotsc, \gamma_n$ be the set of essential pullbacks of $\gamma$ under $f$. The {\em Thurston multiplier} of $\gamma$ is defined to be
\[\lambda_f(\gamma):=\sum_{j=1}^n\frac{1}{\deg(f\from\gamma_j\to\gamma)}.\]
If all pullbacks of $\gamma$ are peripheral, then $\lambda_f(\gamma)=0$.
\end{defn}

\begin{rems}
\,
\begin{enumerate}[leftmargin=2em]
\item By Proposition \ref{curverelA}, the multiplier may be regarded as being defined on isotopy classes.
\item A simple (but as we shall see, remarkably powerful) fact is that a given Thurston map $f$ has finitely many possible multiplier values.
\end{enumerate}
\end{rems}

The notion of the Thurston multiplier is needed to precisely state Thurston's characterization theorem, which gives sharp conditions for when a Thurston map is combinatorially equivalent to a rational Thurston map. We have little use for this theorem since our main results are for maps that are already assumed to be rational, but we nevertheless state it (again, in the restricted setting of $|A|=4$).

An essential Jordan curve $\gamma$ in $(S^2,A)$ is said to {\em $f$-invariant} if each essential pullback of $\gamma$ under $f$ is isotopic to $\gamma \rel A$. Such an $f$-invariant essential Jordan curve is said to be a {\em Thurston obstruction} for $f$ if $\lambda_f(\gamma)\geq 1$.

\begin{thm}[Thurston's criterion]
Let $(f,A)$ be a Thurston map with $|A|=4$ and suppose $f$ has hyperbolic orbifold. Then $f$ is combinatorially equivalent to a rational map if and only if $f$ has no Thurston obstruction.
\end{thm}

The proof for the case where $A=P_f$ was first published in \cite{dhthurston}. A mild generalization for marked Thurston maps (where $P_f\subset A$) can be found in \cite{markedthurston}.

\subsection{Mapping Class Groups}\label{mcgsubsection}

In this section we will sketch some results from the theory of mapping class groups of marked spheres. The standard reference for this material is \cite{fmprimer}, and we will not stray far from presentation laid out there.


\begin{defn}\label{mcgdef} Let $S$ be a Riemann surface (potentially with boundary). Let $\Homeo^+(S,\partial S)$ be the group of orientation-preserving homeomorphisms $\phi\from S\to S$ such that $\phi|_{\partial S}=\id_{\partial S}$; we endow the space with the compact-open topology. A continuous path $\eta\from [0,1]\to \Homeo^+(S,\partial S)$ is equivalent to an isotopy $H\from [0,1]\times S\to S \rel \partial S$ defined by $H_t=\eta(t)$.

Let $\Homeo_0(S,\partial S)$ denote the path-component of $\Homeo^+(S,\partial S)$ that contains the identity map, meaning this set consists precisely of those elements that are isotopic to the identity relative to $\partial S$. This is easily seen to be a normal subgroup of $\Homeo^+(S,\partial S)$, and so the {\em mapping class group} of $S$ is defined to be the quotient
\[\Mod(S):=\Homeo^+(S,\partial S)/\Homeo_0(S,\partial S).\]
\end{defn}

To describe isotopy classes of spheres relative to a finite set of $n$ marked points $A\subset S^2$, it is natural to put $S=S^2\setminus A$ and then consider the corresponding mapping class group. These two notions are not equivalent though, since an element of $\Mod(S^2\setminus A)$ may permute the set of punctures $A$. To amend this, we consider the {\em pure mapping class group}, which is the subgroup of the mapping class group that fixes each puncture individually. Thus we will use $\Homeo^+(S^2, A)$ to denote the group of orientation-preserving homeomorphisms that fix $A$ pointwise and $\PMod(S^2, A)$ to denote the group of isotopy classes of $\Homeo^+(S^2,A)$.

%

The following basic result from the theory of mapping class groups can be found in \cite[Chapter 2]{fmprimer} and also in \cite[Appendix]{buser}.

\begin{prop}\label{PModS03}
$\PMod(S_{0,3})$ is trivial.
\end{prop}

It follows from the preceding proposition that any $\phi\in \Homeo^+(S^2,A)$ for $|A|=3$ is isotopic the identity relative $A$. This is not the case when $|A|\geq 4$, however. Indeed, $\PMod(S_{0,4})\cong F_2$, the free group on two generators. 

To help us characterize which homeomorphisms are relatively isotopic to the identity in the case of four marked points, we will require the following.

\begin{prop}\label{idisotopychar}
Let $A = \{a_1,a_2,a_3,a_4\}\subset S^2$ be a set of four marked points. If $\phi\in\Homeo^+(S^2,A)$ and $\phi$ induces the trivial automorphism on $\pi_1(S^2\setminus \{a_1,a_2,a_3\}, a_4)$, then $\phi$ is isotopic to the identity relative to $A$.
\end{prop}

We omit proof, but this is essentially an instance of the famous ``Alexander method'' (again, see \cite[Chapter 2]{fmprimer}).
 
\section{The Moduli Space Map for $(f,A)$}\label{modmapsection}

In this section we will prove the following theorem:

\begin{thm}\label{maindiagram}
Suppose $(f,A)$ is a rational Thurston map with four marked points $A = \{0,1,\infty,a\}$ and at most three postcritical points $P_f \subset \{0,1,\infty\}=B$ where $f(B)\subset B$. We have the following:
\begin{enumerate}[label={\rm (\roman*)}, leftmargin=2em]
\item If $f(a)\in\{0,1,\infty\}$, then $\sigma_{f,A}$ is constant.
\item If $f(a)=a$, then using the identification $\calM_A = \hC\setminus \{0,1,\infty\}$ and $\calT_A = \bbH$, the following diagram commutes:

\begin{equation}\label{modulimapdiagram}
\begin{tikzcd}
	{\calT_A} && {\calT_A} \\
	\\
	{\calM_A} && {\calM_A,}
	\arrow["{\sigma_{f,A}}"', from=1-3, to=1-1]
	\arrow["f"', from=3-1, to=3-3]
	\arrow["\pi", from=1-3, to=3-3]
	\arrow["\pi"', from=1-1, to=3-1]
\end{tikzcd}
\end{equation}
where $\pi\from\calT_A\to\calM_A$ is the usual holomorphic universal covering map of the three-punctured sphere.
\end{enumerate}
\end{thm}

Note that the $f$ in the above theorem statement is actually its restriction to the image set $\pi(\sigma_{f,A}(\calT_A))$.

We again remark that this particular moduli correspondence diagram has been studied before in \cite{kelseylodgequadratic}. Still, in the interest of keeping this exposition self-contained we will give a detailed and careful proof.

The first part of the theorem can be dispensed with immediately:

\begin{proof}[Proof of part {\rm (i)}]
Suppose $f(a)\in\{0,1,\infty\} = B$. That $\sigma_{f,A}$ is constant will follow from a simple application of ``McMullen's composition trick'' (see \cite{bekp}). We may write $f = f_1\circ f_0$ where $f_0\from (S^2,A)\to (S^2,B)$ is defined by $f_0(z):=f(z)$ and $f_1\from (S^2,B)\to (S^2,A)$ is defined by $f_1(z):=z$. Both of these maps are admissible coverings in the sense described in the previous section, so the pullback composition law Proposition \ref{pullbackcomp} applies and $\sigma_f = \sigma_{f_0}\circ \sigma_{f_1}$. On the other hand, $\sigma_{f_1}\from\calT_A\to\calT_B$ has a single point as its image, so $\sigma_f$ must be constant.
\end{proof}

Before giving the second part of the proof we clarify why the identifications of $\calM_A$ and $\calT_A$ given in the theorem hold.

\subsection{Teichm\"{u}ller Space as a Cover of Moduli Space}\label{modcoversection}

In this section we establish an equivalence between two alternative definitions of Teichm\"{u}ller space: (i) equivalence classes of marking homeomorphisms, and (ii) as the universal cover of moduli space. We will restrict our attention to the case where $|A|=4$ since this is all we need for our purposes, but similar ideas can be used to prove the more general case.

\subsection*{Normalizations.}\addtocontents{toc}{\SkipTocEntry} 

Let us normalize our identifications of $\calT_A$ and $\calM_A$ as in the theorem statement. We wish to work with explicit coordinates, so let $S^2=\hC$. Throughout this section, we will put $A = \{0,1,\infty,a\}$ and $B = \{0,1,\infty\}$. Note that every marking homeomorphism $\phi\from (\hC,A)\to\hC$ is equivalent to one that fixes $B$ pointwise. Indeed, M\"{o}bius transformations are $3$-transitive on the sphere, so we may simply choose $M$ such that $(M\circ\phi)(b) = b$ for each $b\in B$. Then clearly $\phi\sim M\circ \phi$ in the sense of Definition \ref{teichdef}.

Now, as in Section \ref{mcgsubsection}, let ${\rm Homeo}^+(\hC,B)$ denote the set of orientation-preserving marking homeomorphisms $\phi\from \hC\to \hC$ which fix $B$ pointwise. Two elements $\phi_0,\phi_1\in{\rm Homeo}^+(\hC,B)$ are equivalent in $\calT_A$ if and only if there is a M\"{o}bius transformation $M\from\hC\to \hC$ such that $\phi_1\simeq M\circ\phi_0 \rel A$ and $\phi_1=M\circ\phi_0$ on $A$. Since $\phi_1(b)=\phi_0(b)=b$ for each $b\in B$, such a M\"{o}bius transformation $M$ fixes three points and is thus the identity. This allows us to write
\begin{equation}\label{teichnew}
\calT_A = {\rm Homeo}^+(\hC,B)/\sim
\end{equation}
where $\phi_0\sim \phi_1$ if and only if $\phi_0\simeq \phi_1 \rel A$.

We remark that such an isotopy $H_t$ must have $H_t|_A = H_0|_A = \phi_0|_A$. Since $\phi_0$ fixes $B = \{0,1,\infty\}$ pointwise, we get $H_t\in {\rm Homeo}^+(\hC,B)$ and also $H_t(a) = \phi_1(a)=\phi_2(a)$ for each $t\in[0,1]$.

Next, consider the moduli space $\calM_A$. Again appealing to $3$-transitivity of M\"{o}bius transformations, we see that every injection $\iota\from A\hookrightarrow \hC$ is equivalent in $\calM_A$ to one which fixes $B = \{0,1,\infty\}$ pointwise. Within the subset of injections which fix $B$ pointwise, two such injections $\iota_0,\iota_1$ are equivalent if and only if $\iota_0(a)=\iota_1(a)\in\hC\setminus \{0,1,\infty\}$. This yields an identification of $\calM_A$ with $\hC\setminus \{0,1,\infty\}$ given by $[\iota]\mapsto \iota(a)$ where $\iota\from\{a\}\hookrightarrow \hC\setminus \{0,1,\infty\}$. Hence
\begin{equation}\label{modulinew}
\calM_A =\hC\setminus B.
\end{equation}

\subsection*{Universal Cover of Moduli Space}\addtocontents{toc}{\SkipTocEntry} 

The identification (\ref{modulinew}) above is actually a holomorphic isomorphism. Furthermore, $\calT_A$ can also be given a complex structure (see \cite[Section 6.5]{hubbardhandbook}). With respect to this structure, the natural covering map $\pi_A\from\calT_A\to\calM_A$ given by $[\phi]\mapsto \phi(a)$ that we defined in Section \ref{preliminarysection} is actually holomorphic. From these observations it will easily follow that $\calT_A$ is isomorphic to the upper half-plane $\bbH$. Indeed, using our characterizations (\ref{teichnew}) and (\ref{modulinew}), let $\pi_A\from\calT_A\to\calM_A$ be as above and let $\pi\from\bbH\to \calM_A$ be the holomorphic universal covering map provided by the Uniformization Theorem. By invoking universality of the covering spaces (let us say in the topological category for now), there is some homeomorphism $\epsilon\from\calT_A\to \bbH$ such that the diagram below commutes.

\begin{equation}\label{epsilondiagram}
\begin{tikzcd}
	{\calT_A} && \bbH \\
	\\
	& {\calM_A}
	\arrow["{\pi_A}"', from=1-1, to=3-2]
	\arrow["\pi", from=1-3, to=3-2]
	\arrow["\epsilon", from=1-1, to=1-3]
\end{tikzcd}\end{equation}

Since $\pi_A = \pi\circ\epsilon$ and the involved spaces have complex structures such that $\pi_A$ and $\pi$ are holomorphic, we may in fact say $\epsilon$ is a holomorphic isomorphism (rather than merely a homeomorphism). The map $\epsilon$ is unique after specification of a single value. Note $\pi_A([\id_{\hC}]) = a = (\pi\circ \epsilon)([\id_{\hC}])$, so from now on we fix a basepoint $\tau_0\in\pi^{-1}(a)$ so that $\epsilon([\id_{\hC}]) = \tau_0$.

This description alone is technically sufficient for us to proceed with the proof of Theorem \ref{maindiagram}. Still, it is nice to give a more concrete description of the map $\epsilon$. In the remainder of this section we provide an explicit construction of the map $\epsilon$ and check that it establishes the claimed holomorphic isomorphism.

\subsection*{The Construction}\addtocontents{toc}{\SkipTocEntry} 
A representative $\phi$ of an element of $\calT_A$ restricts to a self-map of $\calM_A = \hC\setminus B$. In an abuse of notation, we will not distinguish between $\phi$ and $\phi|_{\calM_A}$. From standard covering space theory there is a lifted homeomorphism $\tphi\from\bbH\to\bbH$ such that the following diagram commutes:
\begin{equation}\label{liftconstruction}
\begin{tikzcd}
	\bbH && \bbH \\
	\\
	{\calM_A} && {\calM_A}.
	\arrow["\pi"', from=1-1, to=3-1]
	\arrow["{\tphi}", from=1-1, to=1-3]
	\arrow["\phi"', from=3-1, to=3-3]
	\arrow["\pi", from=1-3, to=3-3]
\end{tikzcd}
\end{equation}

In general $\tphi$ is not uniquely determined. As we shall see later, though, it becomes so when we add the requirement that $\tphi$ commutes with every deck transformation of $\pi$. We will then prove

\begin{thm}\label{teichidentification}
The map $\epsilon\from \calT_A\to \bbH$ defined by $\epsilon([\phi]) = \tphi(\tau_0)$, where $\tphi$ is the unique lift of $\phi$ as in (\ref{liftconstruction}) which commutes with every deck transformation of $\pi$, is the holomorphic isomorphism satisfying (\ref{epsilondiagram}).
\end{thm}

Before proceeding with the proof we first need some technical results concerning lifts and isotopies.

\begin{lem}[Homotopy lifting property]
Let $\pi\from \widetilde{Y}\to Y$ be a covering map and let $k_0\from X\to Y$ be a map from a locally path-connected space $X$. Let $K\from X\times [0,1]\to Y$ be a homotopy with $K_0 = k_0$, and let $\tilde k_0\from X\to \widetilde{Y}$ be a lift of $k_0$. Then the homotopy $K$ uniquely lifts to a homotopy $\tK\from X\times [0,1]\to \widetilde{Y}$ such that $\tK_0 = \tilde k_0$ and $\pi\circ \tK_t = K_t$ for all $t\in[0,1]$.
\end{lem}

A proof of this fact can be found in any standard reference on algebraic topology, e.g., \cite[Proposition 1.30]{hatcher}. We will use this to prove

\begin{lem}\label{isotopyliftinglemma}
Suppose $\phi\in{\rm Homeo}^+(\hC,B)$ and let $\tphi\from\bbH\to\bbH$ be a lifted homeomorphism as in (\ref{liftconstruction}). Let $H\from \hC\times[0,1]\to \hC$ be an isotopy ${\rm rel.\,} B$ with $H_0 = \phi$. Then the isotopy $H$ uniquely lifts to an isotopy $\tH\from \bbH\times[0,1]\to\bbH$ such that $\tH_0 = \tphi$ and $\pi\circ \tH_t=H_t\circ \pi$ for all $t\in[0,1]$.

Moreover, if $H$ is an isotopy relative to $A=\{0,1,\infty,a\}$, then the lift $\tH$ is an isotopy relative to $\pi^{-1}(a)$.
\end{lem}

\begin{proof}
Consider the map $K\from\bbH\times[0,1]\to\hC\setminus B$ defined by $K_t:= H_t\circ \pi$. This defines a homotopy with $K_0 = \phi\circ\pi$. By definition the map $\tphi\from\bbH\to\bbH$ is a lift of $\phi\circ\pi$ in the sense that $\pi\circ\tphi = \phi\circ\pi$. Thus the homotopy lifting property applies and we obtain a unique lift $\tK\from \bbH\times[0,1]\to \bbH$ such that $\pi\circ\tK_t = K_t = H_t\circ\pi$ for all $t\in [0,1]$ and also $\tK_0 = \tphi$.

We now claim that the map $\tK$ is actually an isotopy. To this end, note that one can also apply the homotopy lifting property to the homotopy $L_t:=H_t^{-1}\circ\pi$ with initial lift $\tphi^{-1}$ to obtain a unique lift $\tL\from \bbH\times[0,1]\to\bbH$ such that $\pi\circ\tL_t = L_t = H_t^{-1}\circ\pi$ for all $t\in[0,1]$ and also $\tL_0=\tphi^{-1}$. Now $\tK_0 \circ \tL_0 = \tphi\circ\tphi^{-1} = \id_\bbH$ and, for each $t\in [0,1]$,
\[\pi\circ \tK_t\circ\tL_t = (H_t\circ \pi)\circ \tL_t  = H_t\circ (\pi\circ \tL_t) = H_t\circ (H_t^{-1}\circ\pi) =\pi.\]
Now for any $\tau\in\bbH$, the continuous path $\alpha(t):=(\tK_t\circ\tL_t)(\tau)$ starts at $\tau$ at $t=0$ and is contained in the discrete set $\pi^{-1}(\pi(\tau))$ by the above observation. By connectedness we conclude $\alpha(t) = \tau$ for all $t\in [0,1]$. Thus, for each $t\in[0,1]$, $(\tK_t\circ\tL_t)(\tau)=\tau$ for each $\tau\in\bbH$ and hence $\tK_t\circ\tL_t = \id_{\bbH}$. This shows that each $\tK_t$ is invertible with continuous inverse $\tK_t^{-1} = \tL_t$, and we conclude $\tK_t$ is a homeomorphism for each $t\in [0,1]$,

Finally, put $\tH_t:=\tK_t$. As we have shown above this is an isotopy with the desired properties, and its uniqueness follows from the uniqueness of the homotopy lifting property.

Let us now prove the last claim of the lemma statement, i.e., let us show that $\tH$ is an isotopy relative to $\pi^{-1}(a)$ given that $H$ is an isotopy relative to $A$. Let $\tau\in\pi^{-1}(a)$ so $\pi(\tau)=a$. Since $H$ is an isotopy relative to $A$, $H_t(a) = \phi(a)$ for all $t\in[0,1]$. Thus
\[\pi(\tH_t(\tau)) = (\pi\circ \tH_t)(\tau) = (H_t\circ\pi)(\tau) =\phi(a).\]
Arguing as before, the path $\beta(t):=\tH_t(\tau)$ begins at $\tH_0(\tau) = \tphi(\tau)$ at $t=0$ and lies inside the discrete set $\pi^{-1}(a)$, from which we conclude $\tH_t(\tau) = \tphi(\tau)$ for all $t\in[0,1]$. This completes the proof.
\end{proof}

The statement of Theorem \ref{teichidentification} makes reference to lifts of $\phi\in\Homeo^+(\hC,B)$ which commute with every deck transformation. The next lemma will establish the existence and uniqueness of such a lift.

\begin{lem}\label{commutinglift}
Suppose $\phi\in\Homeo^+(\hC,B)$. Then $\phi$ has a lift $\tphi$ as in (\ref{liftconstruction}) which commutes with every deck transformation, and this lift is unique.
\end{lem}
\begin{proof}
Let $G$ be the group of deck transformations for $\pi\from \bbH\to\hC\setminus B$. Then there is a group isomorphism $G\cong \pi_1(\hC\setminus B, a)$.  By Proposition \ref{PModS03} we know $\PMod(S^2, B)$ is trivial, so $\phi\simeq\id_{\hC}\rel B$. Let $H$ be this isotopy$\rel B$ starting from the identity and ending at $\phi$, so $H_0 = \id_{\hC}$ and $H_1=\phi$. Since $\pi\circ \id_{\bbH} = \id_{\hC}\circ \pi$, we can use Lemma \ref{isotopyliftinglemma} to lift $H$ to an isotopy $\tH$ with $\tH_0=\id_{\bbH}$.

We now claim that, for any deck transformation $g\in G$ and $t\in[0,1]$, $\tH_t\circ g\circ \tH_t^{-1}$ is also a deck transformation. Indeed,
\begin{align*}
\pi\circ (\tH_t\circ g\circ \tH_t^{-1}) &= (\pi\circ \tH_t)\circ g\circ \tH_t^{-1} \\
& = (H_t\circ \pi)\circ g\circ \tH_t^{-1} \\
& = H_t\circ (\pi\circ g)\circ \tH_t^{-1} \\
& = H_t\circ \pi\circ \tH_t^{-1} = \pi\circ\tH_t\circ\tH_t^{-1} = \pi.
\end{align*}
Thus $\tH_t\circ g\circ \tH_t^{-1}$ is a fiber-preserving homeomorphism and hence a deck transformation for each $t\in[0,1]$. In fact, since the fiber of any point of $\hC\setminus B$ is discrete, it is the same deck transformation for each $t\in[0,1]$. Thus we may write $\tH_t\circ g\circ \tH_t^{-1} = g'$ for $g'\in G$. Yet at $t=0$ we have $\tH_0\circ g\circ \tH_0^{-1} = g$, so $g'=g$. It follows that, at $t=1$, $\tH_1\circ g\circ \tH_1^{-1} = g$. Thus $\tphi:=\tH_1$ is a lift of $\phi$ which commutes with every deck transformation, as desired.

We now show that the lift $\tphi$ is the unique lift of $\phi$ with the property of commuting with every deck transformation. Suppose that $\tphi_1$ and $\tphi_2$ are two such lifts. Since $\tphi_1$ and $\tphi_2$ are lifts of a common map $\phi$, there is a deck transformation $h\in G$ such that $\tphi_1 = h\circ\tphi_2$. Let $g\in G$. On one hand, we have
\[g\circ \tphi_1 = g\circ h\circ \tphi_2.\]
On the other hand,
\[g\circ\tphi_1 = \tphi_1\circ g = h\circ \tphi_2\circ g = h\circ g\circ \tphi_2,\]
where we have used that both $\tphi_1$ and $\tphi_2$ commute with every deck transformation. Thus $(g\circ h)\circ \tphi_2=(h\circ g)\circ \tphi_2$, which implies $g\circ h = h\circ g$ upon pre-composing with $\tphi_2^{-1}$. Because $g\in G$ is arbitrary this implies $h$ is in the center of $G$, i.e., $h\in Z(G) = \{z\in G: zg=gz\text{ for all $g\in G$}\}$. Yet $G\cong\pi_1(\hC\setminus B)\cong F_2$, the free group on two generators, so $Z(G)$ is trivial. Thus $h=\id$ and $\tphi_1=\tphi_2$, as desired.
\end{proof}

\begin{proof}[Proof of Theorem \ref{teichidentification}]
We shall prove $\epsilon$ is a well-defined holomorphic isomorphism in steps.

{\em Step 1}: $\epsilon$ is well-defined. Suppose $\phi_0\sim\phi_1$ in $\calT_A$, meaning $\phi_0\simeq \phi_1\rel A$. Let $\tphi_0$ and $\tphi_1$ be the respective unique lifts which commute with every deck transformation. By Lemma \ref{isotopyliftinglemma} there is a unique isotopy $\tH_t\from \bbH\times [0,1]\to\bbH$ such that $\tH_0=\tphi_0$ and $\pi\circ\tH_t = H_t\circ \pi$ for all $t\in[0,1]$. Since $\tphi_0$ commutes with every deck transformation, so too does $\tH_t$ for each $t\in[0,1]$. Indeed, one can argue as we did in the proof of Lemma \ref{commutinglift} to show that given $g\in G$, the composition $\tH_t\circ g\circ \tH_t^{-1}$ is a fixed deck transformation $g'$ for all $t\in[0,1]$. Since $\tH_0 = \tphi_0$ commutes with $g$ we obtain $g'=g$. It follows that $\tH_1$ is a lift of $\phi_1$ which commutes with every deck transformation, so $\tH_1=\tphi_1$.

Now recall that $\tH_t$ defines an isotopy relative to the fiber $\pi^{-1}(a)$. Since $\tau_0$ is an element of this set, we get $\tH_t(\tau_0)=\tH_0(\tau_0)$ for all $t\in [0,1]$, and thus $\tphi_0(\tau_0) = \tH_0(\tau_0) = \tH_1(\tau_0)=\tphi_1(\tau_0)$. We conclude $\epsilon$ is well-defined.

{\em Step 2}: $\epsilon$ is injective. It will suffice to prove that if $\tphi(\tau_0)=\tau_0$, then $\phi\simeq \id_{\hC}\rel A$. Note $\phi(a)=a$, so $\phi\in\Homeo^+(\hC,A)$ and it induces an automorphism $\phi_*\from\pi_1(\hC\setminus B, a)\to \pi_1(\hC\setminus B, a) $. If we can show that $\phi_*$ is the trivial automorphism, then we can conclude $\phi\simeq \id\rel A$ by Proposition \ref{idisotopychar}. Yet standard covering space theory says we can identify $\pi_1(\hC\setminus B, a)$ with the set of deck transformations $G$,  and using this identification the induced map $\phi_*$ is given by $\phi_*(g) = \tphi\circ g\circ \tphi^{-1}$. Since $\tphi$ was chosen to commute with every deck transformation, $\phi_*(g) = g$ for every $g\in G$. Thus $\epsilon$ is injective.

{\em Step 3}: $\epsilon$ is surjective. Let $\tau\in \bbH$. We must construct a map $\phi\in\Homeo^+(\hC,B)$ such that its lift $\tphi$ has $\tphi(\tau_0)=\tau$. The desired homeomorphism is the one provided by pushing points along a curve connecting the projections of $\tau_0$ and $\tau$ in the punctured sphere $\hC\setminus B$. The construction is a standard exercise, but we still provide a sketch. Put $z = \pi(\tau)\in\hC\setminus B$. Then choose a smooth path $\tgamma$ in $\bbH$ from $\tau_0$ to $\tau$. Then $\gamma:=\pi(\tgamma)$ is a smooth path in $\hC\setminus B$ which connects $a$ to $z$. We may choose a smooth vector field compactly supported in a neighborhood of $\gamma\subset \hC\setminus B$ so that $\gamma$ is an integral curve of the vector field (let the vector field coincide with the tangent vectors of $\gamma$ on $\gamma$, and then extend locally using bump functions). Let $t\mapsto H_t$ be the flow of this vector field on $\hC$. Then this defines an isotopy with the property $H_0=\id_{\hC}$, $H_t(b)=b$ for each $b\in B$, and $H_t(a) = \gamma(t)$ for $t\in[0,1]$. Put $\phi:=H_0$; we claim this will be the desired map.

Since $H_t$ is an isotopy$\rel B$, we may lift it to a unique isotopy $\tH_t\from\bbH\to\bbH$ such that $\tH_0=\id_{\bbH}$. Moreover,
\[\pi(\tH_t(\tau_0)) = (H_t\circ\pi)(\tau_0) = H_t(a) = \gamma(t),\]
so $\tH_t(\tau_0)$ is a lift of $\gamma$ which begins at $\tau_0\in\bbH$. Since $\tgamma$ is also a lift of $\gamma$ beginning at the same point, we must have $\tH_t(\tau_0)=\tgamma(t)$ by the path-lifting property. Thus $\tH_1(\tau_0) = \tgamma(1)=\tau$. Using the same arguments as in Lemma \ref{commutinglift} we can also show $\tH_1$ will commute with every deck transformation, so $\tphi = \tH_1$ by the uniqueness of lifts of $\phi$ with this property. Thus $\epsilon$ is surjective.

{\em Step 4:} the diagram (\ref{epsilondiagram}) holds. Given any $[\phi]\in\calT_A$ we have
\[\pi_A([\phi]) = \phi(a) = \phi(\pi(\tau_0)) = \pi(\tphi(\tau_0)) = (\pi\circ\epsilon)([\phi]).\]
Moreover, $\epsilon$ satisfies the property $\epsilon([\id_{\hC}]) = \tau_0$ by construction.

{\em Step 5:} $\epsilon$ is a holomorphic isomorphism. With respect to the complex structures on $\calT_A$, $\calM_A = \bbC\setminus B$, and $\bbH$, the maps $\pi_A\from\calT_A\to \calM_A$ and $\pi\from\bbH\to\calM_A$ are holomorphic covering maps. Since $\pi_A = \pi\circ\epsilon$, if we can show that $\epsilon$ is a homeomorphism then holomorphicity of $\epsilon$ and its inverse follows by choosing locally defined branches of $\pi^{-1}$ and $\pi_A^{-1}$ respectively. To show that $\epsilon$ is a homeomorphism, it will actually suffice to show that $\epsilon$ is continuous by appealing to the previous step and the uniqueness component of the universal property of covering spaces. For continuity, one must show two classes of markings that are close with respect to the Teichm\"{u}ller metric give rise to lifts that are close at $\tau_0$. This follows from standard (if slightly tedious) arguments, so we omit the details for now.
\end{proof}

\subsection{The Pullback Map for $(f,A)$}

With Theorem \ref{teichidentification} established, we will now identify $\calT_A$ with $\bbH$ via $\epsilon$. Let us briefly review the changes in definition this identification will cause. The definition of the pullback map $\sigma_{f,A}$ changes by conjugation with $\epsilon$,  so now $\sigma_{f,A}$ is defined as $\tphi(\tau_0)\mapsto \tpsi(\tau_0)$ where $\phi,\psi$ are representatives of Teichm\"{u}ller space $\calT_A$ in the sense of (\ref{teichnew}) for which $\sigma_{f,A}([\phi]) = [\psi]$ in the old definition. We also remark $\tau_0\in\bbH$ is the fixed point of $\sigma_{f,A}$ by construction. On the other hand, the natural projection map $\pi_A\from\calT_A\to\calM_A$ will just be the same as the map $\pi\from\bbH\to\hC\setminus\{0,1,\infty\}$ provided by uniformizing.

\begin{proof}[Proof of Theorem \ref{maindiagram} part {\rm (ii)}]
As a reminder, our goal is to show that if $f(a)=a$, then the following diagram commutes:
\begin{equation}\label{modulimapdiagram}
\begin{tikzcd}
	\bbH && \bbH \\
	\\
	\hC\setminus B && {\hC\setminus B.}
	\arrow["{\sigma_{f,A}}"', from=1-3, to=1-1]
	\arrow["f"', from=3-1, to=3-3]
	\arrow["\pi", from=1-3, to=3-3]
	\arrow["\pi"', from=1-1, to=3-1]
\end{tikzcd}
\end{equation}
Suppose $\sigma_{f,A}(\tau) = \tau'$, where $\tau = \tphi(\tau_0)$ and $\tau' = \tpsi(\tau_0)$. Note that
\[(\pi\circ\sigma_{f,A})(\tphi(\tau_0)) = \pi(\tpsi(\tau_0)) = \psi(a).\]
By construction the unlifted representatives $\phi$ and $\psi$ satisfy the commutative diagram of Proposition \ref{pullbackdefdiagram} where we take $F=f$ (since $f$ is already rational), so $\phi\circ f = f\circ \psi$. Combining this with the fact that $f(a)=a$ gives
\[(f\circ \pi\circ \sigma_{f,A})(\tau) = f(\psi(a)) = \phi(f(a)) = \phi(a).\]
Since $\phi(a) = \pi(\tphi(\tau_0)) = \pi(\tau)$, we conclude $
f\circ\pi\circ\sigma_{f,A} = \pi$ for any $\tau\in\bbH$. This completes the proof.
\end{proof}

\section{The Global Curve Attractor for $(f,A)$}\label{gcasection}

\subsection{Equivalence to Attractors on Cusps}

As described in Selinger's thesis \cite{selthesis}, $\sigma_f$ extends to the Weil-Petersson completion of Teichm\"{u}ller space $\overline{\calT}_A$, which for $|A|=4$ is $\overline{\calT}_A= \bbH\cup\hQ$. The behavior of $\sigma_f$ on $\hQ$ encodes the pullback action on essential Jordan curves in $(S^2,A)$. For a detailed discussion of these facts in the case of four marked points, see \cite[Section 6]{NETmaps}.

Maintaining the setting of the previous paragraph, Pilgrim's conjecture on the existence of a finite global curve attractor can be restated for this case as follows:
\begin{conj}\label{gcateich}
Let $f$ be a rational Thurston map with a set $A$ of four marked points and hyperbolic orbifold. Let $\sigma_f\from\overline{\calT}_A\to\overline{\calT}_A$ be the associated pullback map on augmented Teichm\"{u}ller space, where $\overline{\calT}_A = \bbH\cup\hQ$. Then there is a finite set $\calA\subset \hQ$ such that, for every $r\in\hQ$, either $\sigma_f^N(r)\in\bbH$ for some $N$, or $\sigma_f^n(r)\in \calA$ for all $n$ sufficiently large.
\end{conj}

We shall call such a set $\calA$ a {\em finite cusp attractor}.

If $\sigma_f(r)\in \hQ$, then we say that the cusp $r$ has {\em essential pullback}. If $\sigma_f(r)\in\bbH$, then we say that the cusp $r$ has {\em peripheral pullback}.

In this section, we present a proof of the conjecture for the following class of rational Thurston maps:

\resmainthm*

Note that we make no assumptions on the orbifold-type of $f$. This is because $\sigma_{f,A}$ is always contracting under these hypotheses, and will thus have a unique fixed point (ensuring this is the purpose behind the ``hyperbolic orbifold'' assumption in the original formulation).

For the purposes of this section, we will work exclusively in the case where $(f,A)$ has exactly three postcritical points $P_f = \{0,1,\infty\}=B$ and the fourth point of the marking has $f(a)=a$. Handling the other cases admitted by the theorem statement from there is easy, and we describe the necessary adjustments when we give the proof of Theorem \ref{mainthm} at the end of Section \ref{mainthmsubsec}.

\subsection{Horoballs}\label{horoballsubsec}
In this section we record some lemmas regarding horoballs in $\bbH$ for later use. Let $\hR:= \bbR\cup\{\infty\}$ denote the extended real line.

\begin{defn}
Let $t>0$ and $x\in\hR$. A {\em horoball} based at $x$ with {\em hororadius $t$} is
\[
H_{t}(x):= \left\{\tau\in\bbH:\frac{\Im \tau}{|\tau-x|^2}>\frac{1}{t}\right\}
\]
for $x\in\bbR$ and
\[
H_t(\infty): = \left\{\tau\in\bbH: \Im \tau>\frac{1}{t}\right\}
\]
for $x=\infty$.
\end{defn}

A horoball $H_t(x)$ based at a finite point $x\in\bbR$ is simply a ball in $\bbH$ that is tangent to the real line at $x$; a simple calculation shows the ball has Euclidean diameter $t$. The horoball $H_t(\infty)$, on the other hand, is a half-plane.

If $H$ is a horoball, we call its boundary $\partial H$ a {\em horocycle}. Thus a horocycle based at $x\in\bbR$ is a circle in $\bbH$ which is tangent to $\bbR$ at $x$, while a horocycle based at $\infty$ is a horizontal line.

We will now define a special family of normalized horoballs based at rational cusps. Suppose $r=\frac{p}{q}\in\bbQ$ is written in lowest terms. Then put
\[B_t(r) := H_{t/q^2}(r).\]
For $r=\infty=1/0$, we define $B_t(\infty) := H_t(\infty)$. Note the Euclidean radius of $B_t(r)$ for $r\neq\infty$ is given by $R = t/(2q^2)$.

This collection of horoballs is particularly well-behaved under the action of the modular group. Indeed, identify the matrix group $\PSL(2,\bbZ)$ with the subgroup of $\Aut(\bbH)$ given by
\[\left\{\tau\mapsto \frac{a\tau+b}{c\tau+d}:\begin{bmatrix}a & b \\ c & d
\end{bmatrix}\in \SL(2,\bbZ)\right\}.\]
We have the following lemma:
\begin{lem}\label{equalhoro}
Let $g\in\PSL(2,\bbZ)$ and suppose $g(r)=r'$ for $r,r'\in\hQ$. Then
\[g(B_t(r)) = B_t(r').\]
\end{lem}
\begin{proof}

Note that if $r=p/q$ for $p,q$ integers and $q$ positive, then
\[B_t(r) =H_{t/q^2} = \left\{\tau\in\bbH:\frac{\Im\tau}{|q\tau-p|^2}>\frac{1}{t}\right\}.\]
We show that if $g\in\PSL(2,\bbZ)$ and $g(r)=r'=p'/q'$ for $p',q'$ integers and $q'$ positive, then
\[\frac{\Im\tau}{|q\tau-p|^2} = \frac{\Im g(\tau)}{|q'g(\tau)-p'|^2}.\]
It suffices to check the formula for a generating set of $\PSL(2,\bbZ)$, say $g_1$ and $g_2$ where $g_1(\tau)=\tau+1$ and $g_2(\tau) = -1/\tau$. In the first case, note $\Im g_1(\tau) = \Im \tau$ and $p'/q' = p/q+1$, so after putting $p'=p+q$ and $q'=q$ the result follows for $g_1$. Similarly, in the second case we have $\Im g_2(\tau) = \Im \tau/|\tau|^2$ and $p'/q' = -q/p$. Putting $p'=-q$ and $q'=p$, the stated formula is true in this case as well. This completes the proof.
\end{proof}

An easy corollary of this lemma is the fact that, if $t<1$, then the collection $\{B_t(r):r\in\hQ\}$ is disjoint. Suppose for contradiction two horoballs from the family intersect. By application of a suitably chosen $g\in\PSL(2,\bbZ)$, the lemma allows us to assume one of the horoballs is $B_t(\infty) = \{\tau\in\bbH: \Im(\tau)>1/t\}$. Any other horoball $B_t(r)$ for $r\in\bbR$ is tangent to the real line and has Euclidean diameter $t/q^2\leq t<1$, and so has $\Im(\tau)<1$ for all $\tau\in B_t(r)$. Since $1<1/t$, the two horoballs must be disjoint and the claim follows.

We can also say something about the images of horoballs under the Thurston pullback map. For a rational cusp $r\in\hQ$, put $\delta(r) = 1/\lambda(-1/r)$, where $\lambda_f(s)$ is the Thurston multiplier for the curve with rational slope $s$.

\begin{lem}\label{horoballin}
Suppose $\sigma_f(r)=r'$ for $r,r'\in\hQ$. Let $\delta = \delta(r)$ as above. Then
\[
\sigma_f(B_t(r))\subset B_{\delta t}(r').
\]
\end{lem}

We omit proof of the above lemma, but it can be found in \cite[Section 6]{NETmaps}. The argument presented there establishes a relationship between horoballs and the moduli of the associated curve families, and then appeals to the subadditivity of moduli. This is not the only perspective though\textemdash one can also prove the above lemma by using angular derivatives. In fact, many properties of the Thurston pullback map can be reformulated via the classical Julia-Wolff-Carath\'{e}odory theory. The author plans to present this viewpoint in a separate paper.

\subsection{Cusp Types}

For a Thurston map, if $p\in P_f$ then $f^n(p)\in P_f$ for all integers $n\geq 0$. Since $P_f$ is a finite set the iterates are eventually periodic. It follows that there are minimal integers $k\geq 0$ and $m\geq 1$ such that $f^{k+m}(p)=f^k(p)$. If $q=f^k(p)$, then the cycle associated to $p$ (by which, we mean the cycle forward iterates of $p$ eventually lie in) is given by $\{q,f(q),\dotsc, f^{m-1}(q)\}$.

\begin{defn}
If the cycle constructed above contains a critical point, we say $p$ is a {\em Fatou postcritical point}. If the cycle does not contain a critical point, we say $p$ is a {\em Julia postcritical point}. The set of Fatou postcritical points is denoted $P_\calF$, while the set of Julia postcritical points is denoted $P_\calJ$.
\end{defn}

Obviously $P_f=P_\calF\cup P_\calJ$. The terminology is a consequence of the fact that the points of $P_\calF$ belong to the Fatou set $\calF$, and similarly the points of $P_\calJ$ belong to the Julia set $\calJ$.

\begin{exmp}\label{quadoriginalexmp}
Consider the quadratic map $g(z)=(1-2z)^2$. This map has the following dynamical portrait:

\[\begin{tikzcd}
	{\frac{1}{2}} & 0 & 1\ar[loop, out=-30, in=30, distance=20] && \infty\ar[loop, out=-30, in=30, distance=20, "2:1"']
	\arrow["{2:1}", from=1-1, to=1-2]
	\arrow[from=1-2, to=1-3]
\end{tikzcd}\]
The postcritical set is $P_g = \{0,1,\infty\}$. Evidently, $P_\calF = \{\infty\}$ (since this is a critical fixed point) while $P_\calJ = \{0,1\}$ (since the cycle $1\mapsto 1$ does not contain a critical point).
\end{exmp}

Let us now specialize to the case of $(f,A)$ satisfying Theorem \ref{maindiagram} with $f(a)=a$, so the following diagram commutes:

\begin{equation}
\begin{tikzcd}
	\bbH && \bbH \\
	\\
	\hC\setminus P_f && {\hC\setminus P_f.}
	\arrow["{\sigma_{f}}"', from=1-3, to=1-1]
	\arrow["f"', from=3-1, to=3-3]
	\arrow["\pi", from=1-3, to=3-3]
	\arrow["\pi"', from=1-1, to=3-1]
\end{tikzcd}
\end{equation}
where $P_f = \{0,1,\infty\}$. Each rational cusp of $\hQ$ corresponds to one of these three points. Indeed, if $\gamma\from [0,1)\to\bbH$ is a path that ends nontangentially at $r\in\hQ$ in the sense that $\lim_{t\to 1^-}\gamma(t) = r$, then its projection $\overline{\gamma}:=\pi\circ \gamma$ must have $\lim_{t\to 1^-}\overline{\gamma}(t) = p$ for one of the points $p\in P_f$, and this point $p$ only depends on $r$ and not on the choice of path $\gamma$.

\begin{defn} If $p$ is a Fatou postcritical point, we will call $r$ a {\em Fatou cusp}. If $p$ is a Julia postcritical point, we will call $r$ a {\em Julia cusp}. The sets of Fatou cusps and Julia cusps will be denoted by $\calC_\calF$ and $\calC_\calJ$ respectively.
\end{defn}

Clearly these sets have the property $\calC_\calF\cap\calC_\calJ = \emptyset$ and $\hQ = \calC_\calF\cup\calC_\calJ$. They also satisfy the crucial property that they are not mixed by $\sigma_f$. More precisely:

\begin{lem}\label{cuspsorting}
Suppose $r\in \hQ$ has essential pullback, so $\sigma_f(r)=r'\in\hQ$. If $r$ is a Fatou cusp, then so is $r'$. Similarly, if $r$ is a Julia cusp, so is $r'$.
\end{lem}
\begin{proof}
This easily follows from the fundamental diagram, since Fatou postcritical points only map to Fatou postcritical points, and likewise Julia postcritical points only map to Julia postcritical points. 
\end{proof}

\subsection{The Orbifold Metric}\label{orbifoldsubsec}
In the next few paragraphs we introduce some tools from the theory of orbifolds which we will need for our proof of the main theorem. For general background on orbifolds and orbifold metrics for rational Thurston maps, we refer the reader to \cite[Appendix A.10]{bmexpanding}, \cite[Appendix E]{milnorcd}, or \cite[Appendix A]{ctm}.

For any rational Thurston map $f\from\hC\to\hC$, there is a conformal metric $ds = \rho(z)|dz|$ defined on $\hC\setminus P_\calF$ with $\rho$ smooth on $\hC\setminus P_f$ with the following special property: for any $w\in \hC\setminus P_\calF$ and $z\in f^{-1}(w)$, it satisfies $\norm{df_z}_\rho>1$. It follows that for any compact set $K\subset\hC\setminus P_\calF$ there is a constant $c_K>1$ such that $\norm{df_z}_\rho\geq c_K$ for all $z\in f^{-1}(K)$, i.e., $f$ is expansive on the preimage of $K$ with respect to this metric.

The conformal metric $\rho$ described above is constructed by pushing forward the metric on an orbifold cover. We remark that we do not always use the canonical universal orbifold cover associated to a Thurston map $f$ for this purpose. Indeed, consider $(\hC,\nu_f)$ where $\nu_f$ is the ramification function associated to the Thurston map $f$ introduced in Section \ref{preliminarysection}. When $(\hC,\nu_f)$ is hyperbolic we can (and will) use the associated orbifold metric for $\rho$. If $(\hC,\nu_f)$ is not hyperbolic, then we modify the ramification function so that it is. In particular, if $(\hC,\nu_f)$ has Euclidean orbifold, then we define $\nu\from\hC\to \bbN$ by $\nu(p) = 2\nu_f(p)$ if $\nu_f(p)\geq 2$ and $\nu(p)=1$ otherwise. Since $|{\rm supp}\,(\nu_f)| = |P_f|=3$, this new orbifold $(\hC,\nu)$ is necessarily hyperbolic, and the metric induced by the associated orbifold cover will satisfy the properties we desire.

\subsection{Leashing and Proof of Theorem \ref{mainthm}}\label{mainthmsubsec}

We are now ready to begin presenting the proof of the main theorem.

First we will identify a subset $X\subset \bbH$ of hyperbolic space which is (i) forward-invariant under $\sigma_f$, meaning $\sigma_f(X)\subset X$; (ii) truncated, meaning that its complement in $\bbH$ is a collection of disjoint horoballs; and (iii) satisfies special contractive properties. The horoballs in the complement of $X$ will be the ones we ``leash'' to the fixed point of $\sigma_f$ later in the proof.

\begin{lem} For $t>0$ sufficiently small, there is a subset $X\subset \bbH$ such that $\pi(X)\cup P_{\calJ}$ is compact in $\hC\setminus P_\calF$, $\sigma_f(X)\subset X$, and $\bbH\setminus X$ is the countable disjoint union of all sets of the form $B_t(r)$ where $r\in\hQ$ is a Fatou cusp.
\end{lem}
\begin{proof}

Consider the fundamental diagram
\[
\begin{tikzcd}
	\bbH && \bbH \\
	\\
	{\hC\setminus P_f} && {\hC\setminus P_f.}
	\arrow["\pi"', from=1-1, to=3-1]
	\arrow["{\sigma_f}"', from=1-3, to=1-1]
	\arrow["f"', from=3-1, to=3-3]
	\arrow["\pi", from=1-3, to=3-3]
\end{tikzcd}
\]
If $P_\calF = \emptyset$, then $X=\bbH$ satisfies all the statements of the lemma. Thus suppose $P_\calF\neq\emptyset$, and put $n=|P_\calF|\leq 3$.

Fix some $t<1$, and for each $p_k\in P_\calF$ choose some corresponding Fatou cusp $r_k\in\calC_\calF$. Put $U_k = \pi(B_t(r_k))\cup\{p_k\}$, so that $U_k$ is an open neighborhood of $p_k$ in $\hC$. By shrinking $t$ (and hence the sets $U_k$) as necessary, we can make the following assumptions: (i) each $U_k\setminus \{p_k\}$ is evenly covered by $\pi$; (ii) the sets $U_k$ are pairwise disjoint topological disks whose closures in $\hC$ are disjoint from $P_\calJ$; and (iii) the collection satisfies
\[f(U_1\cup\dotsm\cup U_n)\subset U_1\cup\dotsm\cup U_n.\]
The first two claims are obvious, and the third follows from the fact that the points of $P_\calF = \{p_1,\dotsc, p_n\}$ all lie in superattracting cycles.

We will now show that, for each $k=1,\dotsc, n$, every component of $\pi^{-1}(U_k)$ has the form $B_t(r)$ for some $r\in\calC_\calF$. Indeed, let $B_1,B_2$ be two components of $\pi^{-1}(U_k\setminus\{p_k\})$. Then these components satisfy $g(B_1)=B_2$ for some deck transformation $g$ of $\pi$. On the other hand, $\Deck(\pi)= \overline{\Gamma}(2)\leq \PSL(2,\bbZ)$, so Lemma \ref{equalhoro} shows that since at least one component of $\pi^{-1}(U_k\setminus\{p_k\})$ has the claimed form, they all do. Conversely, every $B_t(r)$ for $r\in\calC_\calF$ is a component of some $\pi^{-1}(U_k\setminus\{p_k\})$. Thus we may put

\[X: = \bbH\setminus \bigg(\bigcup_{r\in\calC_\calF}B_t(r)\bigg).\]

The space $X$ has the property that
\[\pi(X)\cup P_\calJ = \hC\setminus (U_1\cup\dotsm\cup U_n).\]
This is a sphere minus a finite number of open topological disks containing each puncture of $\hC\setminus P_\calF$, and is hence compact in this space.

We now prove that $\sigma_f(X)\subset X$. Let $\tau\in X$. By construction if $f(z)\notin U_1\cup\dotsm\cup U_n$, then $z\notin U_1\cup\dotsm\cup U_n$. The projection satisfies $\pi(\tau)\notin U_1\cup\dotsm\cup U_n$, yet $\pi(\tau) = f(\pi(\sigma_f(\tau)))$, so $\pi(\sigma_f(x))\notin U_1\cup\dotsm\cup U_n$. It follows that $\sigma_f(\tau)\in X$. This completes the proof.
\end{proof}

Consider the orbifold metric $ds = \rho(z)\,|dz|$ introduced earlier in this section, which is defined on $\hC\setminus P_\calF$ and smooth on $\hC\setminus P_f$. We may pull this metric back to a metric $d\tilde{s} = \pi^*(ds) =: \trho(\tau)\,|d\tau|$ that is smooth on $\bbH$. This metric is not in general the same as the standard hyperbolic metric. Indeed, $(\bbH, \trho)$ is incomplete, and some points on the boundary (specifically, the Julia cusps) are now a finite distance away from points in the interior. The completion of $\bbH$ with respect to $\trho$ is $\bbH\cup\calC_\calJ$.

\begin{lem}\label{sigmacontracts}
For all $\tau\in\bbH$, $\norm{d(\sigma_f)_\tau}_{\trho}<1$. Moreover, there is some $0\leq \alpha<1$ such that $\norm{d(\sigma_f)_\tau}_{\trho}\leq \alpha$ for all $\tau\in X$.
\end{lem}
\begin{proof}
Again recall our fundamental diagram, which we now view as a diagram of Riemannian manifolds:
\[
\begin{tikzcd}
	(\bbH,\trho) && (\bbH,\trho) \\
	\\
	{(\hC\setminus P_f, \rho)} && {(\hC\setminus P_f,\rho).}
	\arrow["\pi"', from=1-1, to=3-1]
	\arrow["{\sigma_f}"', from=1-3, to=1-1]
	\arrow["f"', from=3-1, to=3-3]
	\arrow["\pi", from=1-3, to=3-3]
\end{tikzcd}
\]
Fix some $\tau\in\bbH$, and put $\tau'=\sigma_f(\tau)$, $z=\pi(\tau')$ and $w=\pi(\tau)$. We get an induced diagram on tangent spaces:
\[
\begin{tikzcd}
	T_{\tau'}\bbH && T_\tau\bbH \\
	\\
	{T_z(\hC\setminus P_f)} && {T_w(\hC\setminus P_f).}
	\arrow[from=1-1, to=3-1]
	\arrow["{d(\sigma_f)_\tau}"', from=1-3, to=1-1]
	\arrow["{df_z}"', from=3-1, to=3-3]
	\arrow[from=1-3, to=3-3]
\end{tikzcd}
\]
Going around this diagram it is easy to see $\norm{d(\sigma_f)_\tau}_{\trho}=\norm{df_z}_\rho^{-1}$, so the first claim follows. The second claim is a consequence of the fact that $\pi(X)\cup P_\calJ$ is compact in $\hC\setminus P_\calF$.
\end{proof}

Let $\bbH^* = \bbH\cup\calC_\calJ$ denote the completion of $\bbH$ with respect to $\trho$. For any smooth curve $\gamma\from[0,1]\to \bbH^*$ we may write
\[l_{\trho}(\gamma) = \int_\gamma \trho(\tau)\,|d\tau|.\]

Note that if $\gamma$ is a smooth path in $X^* = X\cup\calC_\calJ$ and $\gamma' = \sigma(\gamma)$, then $l_\trho(\gamma')\leq \alpha l_\trho(\gamma)$ by Lemma \ref{sigmacontracts}. We also know that $\gamma'\subset X^*$ since $X$ was chosen to be $\sigma_f$-invariant. With these nice properties in mind, consider the following path metric induced by $\trho$ on $X^*$:
\[d_{X^*}(x_1,x_2) = \inf\{l_\trho(\gamma): \gamma\from[0,1]\to X^* \text{ is a smooth path connecting $x_1$ to $x_2$}\}.\]
This metric is well-defined since $X^*=X\cup\calC_\calJ$ is path-connected, and by construction $\sigma_f$ is uniformly contracting on $(X^*,d_{X^*})$.

In the interest of avoiding notational clutter, we will simply write $d(\cdot,\cdot) = d_{X^*}(\cdot,\cdot)$ for the metric on $X^*$, and $l(\cdot) = l_{\trho}(\cdot)$ for lengths as measured by $\trho$ for the remainder of this section.

Let $\delta_0 = 1/(\min \lambda_f(s))$ where $\lambda_f$ is the Thurston multiplier associated to pulling back a curve with rational slope $s$ by $f$. The minimum is taken over all curves with essential pullback, so $\lambda_f\neq 0$. There are only finitely many possible multiplier values, so the minimum indeed exists and $\delta_0$ is a positive finite number.

\begin{lem}
For $1\leq \delta\leq\delta_0$ and $t>0$ sufficiently small, there is some $c\geq 0$ depending only on $f$ so that
\[d(\partial B_t(r), \partial B_{\delta t}(r))\leq c\ln|\delta|\]
for all Fatou cusps $r\in\hQ$.
\end{lem}
\begin{proof}
We just provide a sketch of the proof. Take a path $\gamma$ connecting points of $\partial B_t(r)$ and $\partial B_{t\delta}(r)$, and then project down to $\hC\setminus P_f$ by $\pi$. Near Fatou postcritical points the Riemannian metric $\rho$ is comparable to the standard hyperbolic metric in a punctured disk where the Fatou postcritical point corresponds to the puncture (see, e.g., \cite[Proposition A.33]{bmexpanding}). By choosing $t>0$ sufficiently small we may ensure all horoballs $B_{\delta_0t}$ project into these neighhorhoods. The horocycles project to concentric circles in this local punctured disk, and it is not hard to calculate the hyperbolic distance between the concentric circles as $|\ln(\delta)|$.
\end{proof}

We now have everything in place to prove the following two key estimates:

\begin{lem}\label{keyineqjulia}
For all Julia cusps $r,r'\in\hQ$ with $\sigma_f(r)=r'$,
\[d(\tau_0,r')\leq\alpha d(\tau_0,r),\]
where $\alpha$ is as in Lemma \ref{sigmacontracts}.
\end{lem}
\begin{proof}
This is immediate from the construction of the metric $d$.
\end{proof}

\begin{lem}\label{keyineqfatou}
There is some constant $C\geq 0$ depending only on $f$ such that, for all Fatou cusps $r,r'\in\hQ$ with $\sigma_f(r)=r'$,
\[d\big(\tau_0,\partial B_t(r')\big)\leq \alpha d\big(\tau_0,\partial B_t(r)\big)+C.\]
\end{lem}

\begin{proof}
Let $r\in\hQ$ be a Fatou cusp such that $\sigma_f(r)\in\hQ$. Then $r' = \sigma_f(r)$ is also a Fatou cusp. Let $\epsilon>0$, and take a path $\gamma$ from $\tau_0$ to some point on $\partial B_t(r)$ such that $l(\gamma)<d(\tau_0,\partial B_t(r))+\epsilon$.

Note that for $\delta := \delta(r)$, we have
\[\sigma_f(B_t(r))\subset B_{\delta t}(r').\]
Put $\gamma' = \sigma_f(\gamma)$. There are two cases to consider:

{\em Case 1}: $\delta\leq 1$. In this case $B_{\delta t}(r')\subset B_{t}(r')$, and thus $\sigma_f(B_t(r))\subset B_t(r')$. We then get
\[d\big(\tau_0, \partial B_t(r')\big)\leq d\big(\tau_0, \sigma_f(\partial B_t(r))\big)\leq l(\gamma')\leq \alpha l(\gamma) < \alpha \big(d(\tau_0,\partial B_t(r))+\epsilon\big).\]
Since
\[d\big(\tau_0, \partial B_t(r')\big)\leq \alpha \big(d(\tau_0,\partial B_t(r))+\epsilon\big)\]
is true for all $\epsilon>0$, we have the desired inequality for any choice of $C$.

{\em Case 2}: $\delta>1$. Then
\begin{align*}
d(\tau_0,\partial B_t(r')&\leq l(\gamma')+d(\partial B_{\delta t}(r'), \partial B_t(r')) \\
&\leq \alpha l(\gamma)+d(\partial B_{\delta t}(r'), \partial B_t(r')).
\end{align*}
On the other hand, $d(\partial B_{\delta t}(r'), \partial B_t(r')) \lesssim |\ln(\delta)|$.

\begin{figure}[H]
\centering
\includegraphics[height=10em]{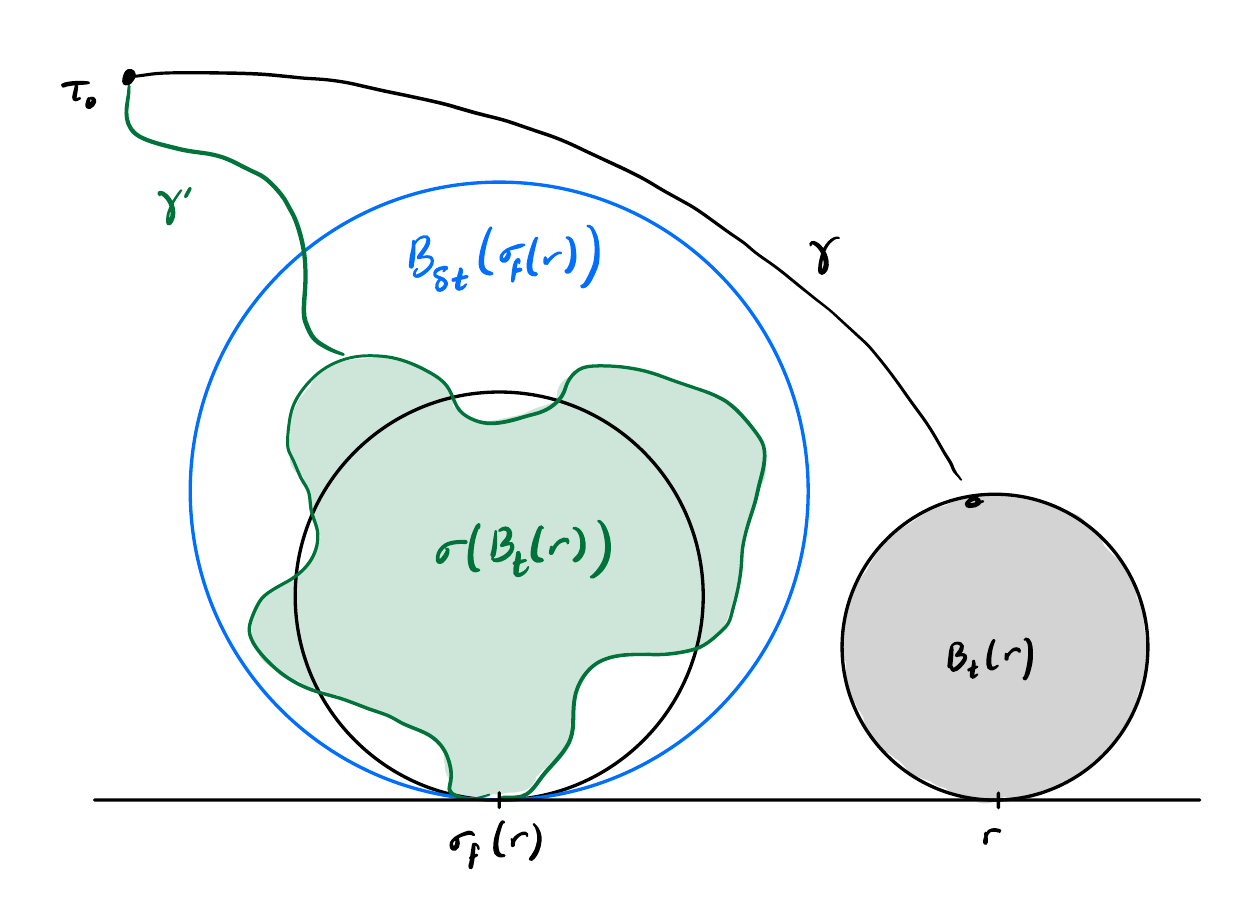}
\caption{Generic picture for $\delta>1$ case.}
\end{figure}

Since $\delta(r) = 1/\lambda_f(-1/r)$ where $\lambda_f$ is the Thurston multiplier, there are only finitely many values it can take. Thus there is a constant $C>0$ so that $|\ln(\delta(r))|\leq C$ for all Fatou cusps $r\in\hQ$ with essential pullback. It follows that
\[d\big(\tau_0,\partial B_m(r')\big)\leq \alpha (d\big(\tau_0,\partial B_t(r)\big)+\epsilon)+C\]
for all $\epsilon>0$, so we get the desired inequality in this case as well.
\end{proof}

\begin{lem}\label{seqlemma}
Let $(x_n) = (x_0, x_1,x_2,\dotsc)$ be a sequence of nonnegative real numbers with the property that there is some $\alpha\in (0,1)$ and $C\geq 0$ for which
\[x_n\leq \alpha x_{n-1}+C\]
for all $n\geq 1$. Then for each $\epsilon>0$ there is an $N$ so that
\[x_n\leq\frac{C}{1-\alpha}+\epsilon\]
for all $n\geq N$.
\end{lem}
\begin{proof}
Observe that
\[x_2\leq \alpha x_1+C\leq  \alpha(\alpha x_{0}+C)+C = \alpha^2 x_{0}+\alpha C+C.\]
Continuing inductively, we get
\begin{align*}
x_{n}\leq \alpha^n x_0+C\sum_{k=0}^{n-1}\alpha^k &= \alpha^nx_0+C\frac{1-\alpha^n}{1-\alpha} = \frac{C}{1-\alpha}+\alpha^n\left(x_0 - \frac{C}{1-\alpha}\right)
\end{align*}
for $n\geq 1$. Thus, after taking absolute values we have the estimate
\[
x_n \leq \frac{C}{1-\alpha}+\alpha^n\left|x_0-\frac{C}{1-\alpha}\right|.\]
Since $\alpha^n\to 0$ as $n\to\infty$, the claimed result follows.
\end{proof}

\begin{lem}\label{locfinite}
A disjoint collection of nonempty open horoballs in $\bbH$ is locally finite.
\end{lem}
\begin{proof}
This result is easier to see in the Poincar\'{e} disk model of the hyperbolic plane. In this setting, a compact set $K$ in $\bbD$ must be contained inside some Euclidean disk $D(0,r)$ with $r<1$. Suppose that a horoball intersects $K$, and hence also $D(0,r)$. Then the diameter of the horoball is at least $1-r$, its Euclidean radius is at least $(1-r)/2$, and its Euclidean area is at least $\pi(1-r)^2/4$. Since $\bbD$ has finite Euclidean area $\pi$, there can be at most finitely many disjoint horoballs in the collection which intersect $K$. This completes the proof.
\end{proof}

\begin{thm}
The map $\sigma_f\from\bbH\to\bbH$ has a finite global cusp attractor. Moreover, every Julia cusp eventually becomes peripheral.
\end{thm}
\begin{proof}
Using Lemma \ref{cuspsorting}, we divide the proof in two by separately considering the pullback action on Julia cusps and Fatou cusps.

Let us begin with the Julia cusps. Let $r\in\hQ$ be a Julia cusp, and suppose for contradiction that $\sigma_f^n(r)$ is not peripheral for any $n\in\bbN$. Then
\[d(\tau_0,\sigma_f^n(r))\geq d_{\trho}(\tau_0,\sigma_f^n(r))) = d_\rho(a, \pi(\sigma_f^n(r)))\geq \min_{p\in P_\calJ} d_\rho(a,p)>0.\]
On the other hand, by Lemma \ref{keyineqjulia} we know
\[d(\tau_0,\sigma_f^n(r))\leq \alpha^n d(\tau_0, r)\to 0\]
as $n\to\infty$, a contradiction. Thus every Julia cusp is eventually peripheral under iteration by $\sigma_f$.

Now suppose $r\in\hQ$ is a Fatou cusp. If there is $n\in\bbN$ such that $\sigma_f^n(r)\in\bbH$, then we are done. Otherwise $\sigma_f^n(r)$ is a Fatou cusp for every $n\in\bbN$, and we can define the sequence
\[x_n = d(\tau_0,\partial B_t(\sigma_f^n(r))) \]
for all $n\geq 1$, and we put $x_0 = d(\tau_0,r)$. By Lemma \ref{keyineqfatou}, this sequence satisfies the hypotheses of Lemma \ref{seqlemma}. Thus there is some $N(r)$ so that
\[x_n := d(\tau_0,\partial B_t(\sigma_f^n(r))\leq \frac{C}{1-\alpha}+1\]
for all $n\geq N(r)$. Yet the set
\[K = \bigg\{\tau\in \bbH: d(\tau_0,\tau)\leq \frac{C}{1-\alpha}+1\bigg\}\]
is a compact subset of $X^* = X\cup \calC_\calJ$, and so it is not hard to see that that the set
\[K' = K\setminus \bigcup_{r\in\calC_\calJ}B_t(r)\]
is a compact subset of $\bbH$. It is clear that $B_t(r)$ for $r\in\calC_\calF$ intersects $K$ if and only if it intersects $K'$, so $K$ intersects only finitely many horoballs from the collection $\{B_t(r): r\in \calC_\calF\}$ by Lemma \ref{locfinite}. Let $\calA = \{r_1,\dotsc, r_k\}$ be the base cusps of this finite collection of horoballs. Then it must be the case $\sigma_f^n(r)\in \calA$ for $n\geq N(r)$. Since $C,\alpha$ and hence $\calA$ depend only on the map $f$ and not on the initial cusp $r$, the set $\calA$ is our desired finite cusp attractor.
\end{proof}

\begin{proof}[Proof of Theorem \ref{mainthm}]
By the equivalence between curve attractors and cusp attractors discussed at the beginning of the section, the above result establishes the theorem in the special case where $(f,A)$ has exactly three postcritical points $P_f=\{0,1,\infty\}$ and $f(a)=a$ is fixed. The case where $a$ is not fixed is handled by statement (i) in Theorem \ref{maindiagram}; if $\sigma_f$ is simply a constant then every curve has peripheral pullback.

Since every rational Thurston map has at least two postcritical points, the only remaining case to consider is when there are exactly two postcritical points. If $f(B)\subset B$ for $B = \{0,1,\infty\}$ still, then the proof goes through with minimal changes. As before, we can assume $f(a)=a$ for otherwise $\sigma_f$ is constant. We then treat $b\in B\setminus P_f$ as a Julia postcritical point, and when we construct the altered ramification function $\nu\from\hC\to\bbN$ which gives rise to our metric $\rho$, we will put $\nu(b)=2$ to ensure hyperbolicity of the orbifold.

Now consider the case where $f$ has two postcritical points but does not satisfy $f(B)\subset B$. Then there is a single point $b\in B\setminus P_f$ so that $f(b)=a$. If $f(a)=a$, then we may simply swap the labels on $a$ and $b$ so that $f(B)\subset B$ and the previous case applies. If, on the other hand, $f(a)\in B$, then $f^2(B)\subset B$ and we replace $f$ with $f^2$ in the proof. This will be sufficient for the theorem because $(f,A)$ has a finite global curve attractor if and only if $(f^k,A)$ does for each iterate $k\in\bbN$.
\end{proof}

\subsection{Examples}
In this section we give some examples of Thurston maps for which our main result applies.

\begin{exmp}\label{Lexample}
Consider a Thurston map with the following combinatorial description: glue two copies of the unit square together along their boundary to obtain a ``pillow'' as depicted in the figure on the right. This is a topological $2$-sphere. Choose the four corners $\{a,b,c,d\}$ as a marking set $A$. On the left, we have a tiled surface which is also a topological $2$-sphere. We define a Thurston map $g$ from the left sphere to the right sphere by mapping the white tile on the left to the front face of the pillow on the right, and then extending by Schwarz reflection. Thus the grey tiles map to the back face of the pillow while the white tiles map to the front face of the pillow.
\vspace{-1em}
\begin{figure}[H]
\centering
\begin{tikzpicture}
	\node at (0,0) {\includegraphics[height=13.65em]{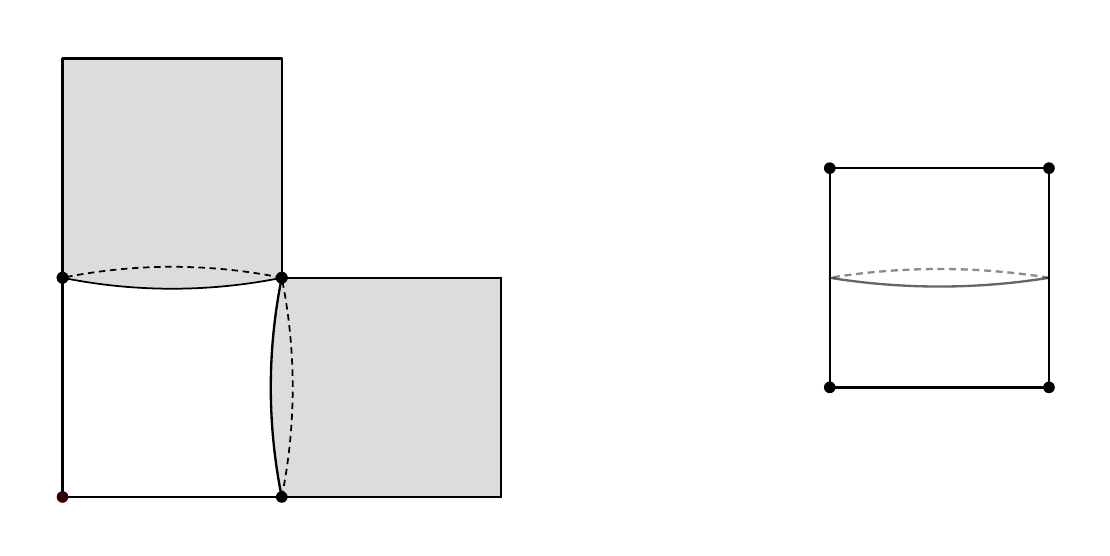}};
	\draw[->](0,0) to node[midway,above] {$f$} (2,0);
	\node at (-4.5, -2.2) {$a$};
	\node at (-2.2, -2.2) {$b$};
	\node at (-2.2, 0.2) {$c$};
	\node at (-4.5, 0.2) {$d$};
	
	\node at (2.2, -1.15) {$a$};
	\node at (4.45, -1.15) {$b$};
	\node at (4.45, 1.15) {$c$};
	\node at (2.2, 1.15) {$d$};
\end{tikzpicture}
\caption{The map $(g,A)$ of Example \ref{Lexample}}
\end{figure}
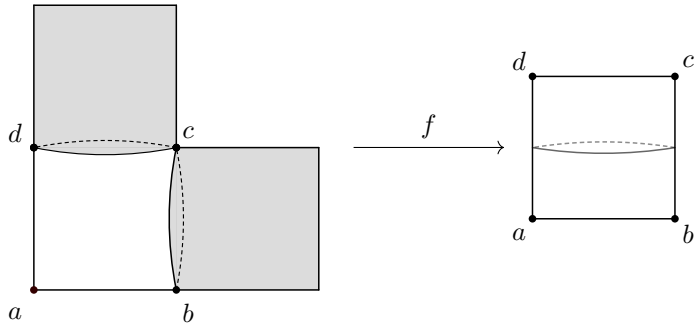

The map $g$ is a degree $3$ Thurston map with three fixed postcritical points $P_g = \{b,c,d\}$ and a non-postcritical fixed point $a$. One can show this map has no Thurston obstruction by using the intersection number techniques of \cite{bhi} (see, in particular, the proofs of Lemma 5.1 and Theorem 5.2 of this paper). Since its orbifold signature is not $(2,2,2,2)$, the marked variant of Thurston's theorem implies $(g,A)$ is combinatorially equivalent to a rational map (see \cite[Theorem 2.1]{markedthurston}). Hence $g$ satisfies the hypotheses of Theorem \ref{mainthm} and has a finite global curve attractor.
\end{exmp}

\begin{exmp}\label{cubicexample}
Let us now consider the rational map
\[g(z) = -\frac{(z-1)^3}{(3z+1)^2}.\]
This map has dynamical portrait

\[\begin{tikzcd}
	{-3} & 1 & 0 & {-\frac{1}{3}} & \infty\ar[loop, out=-30, in=30, distance=20]
	\arrow["{2:1}", from=1-1, to=1-2]
	\arrow["{3:1}",bend left, from=1-2, to=1-3]
	\arrow[from=1-3, bend left, to=1-2]
	\arrow["{2:1}", from=1-4, to=1-5]
\end{tikzcd}\]
The postcritical set is $P_g = \{0,1,\infty\}$ and $g$ also has fixed points $\frac{1}{5}, -\frac{1}{4}\pm i\frac{\sqrt{7}}{4}$, so after selecting $a$ from this three options, $(g,A)$ with $A = \{0,1,\infty,a\}$ satisfies the theorem and has a finite global curve attractor.

Note that $P_\calF = \{0,1\}$ and $P_\calJ = \{\infty\}$ for this particular map, so it possesses a mixture of attracting and repelling dynamics on moduli space. Accordingly this map cannot be shown to have an attractor by means of \cite[Theorem 7.2]{KPSlifts}, which we present and discuss in Section \ref{futuresection}.
\end{exmp}

\section{Extensions to Four Postcritical Points}\label{extensionsection}

In this section we describe how the main theorem can be extended to rational maps which may have as many as four postcritical points. The key observation is the following (somewhat obvious) fact:
\begin{prop}\label{compattractor}
Suppose $(g,A)$ is a rational map which has a finite global curve attractor. If $(h,A)$ is any rational map for which $\sigma_h$ is the identity, then the compositions $h\circ g$ and $g\circ h$ both have finite global curve attractors.
\end{prop}
\begin{proof}
Let $f=h\circ g$. By the pullback composition law, $\sigma_{f} = \sigma_g\circ\sigma_h = \sigma_g$. Since $\sigma_g$ has a finite cusp attractor, the claim follows. The proof for the other composition order is analogous.
\end{proof}

There are two important classes of marked rational maps $(h,A)$ which are known to satisfy $\sigma_h=\id$. Namely, flexible Latt\`{e}s maps and M\"{o}bius transformations. Let us consider examples for both cases.

\begin{exmp}

For $k\in \hC\setminus \{0,1,-1,\infty\}$, the family
\[h_k(z) = \frac{4z(1-z)(1-k^2z)}{(1-k^2z^2)^2}\]
consists of flexible Latt\`{e}s  maps (see \cite[Section 3.6]{bmexpanding}). The postcritical set is $P_{h_k} = \{0,1,\infty, k^{-2}\}$, so for $k=\sqrt{5}$ we obtain a marked rational map $h:= h_{\sqrt{5}}$ with $P_h = A = \{0,1,\infty,a\}$ and $a=1/5$. If we also use $a=1/5$ in the marking for the map $g$ in Example \ref{cubicexample}, then the resulting composition $f = h\circ g$ is a rational Thurston map with postcritical set $P_f = A$. By Proposition \ref{compattractor}, the map $(f,P_f)$ has a finite global curve attractor.
\end{exmp}

\begin{exmp}\label{twistexmp}
Consider again the quadratic map $g(z)=(1-2z)^2$ which we previously saw in Example \ref{quadoriginalexmp}. It had the following dynamical portrait:

\[\begin{tikzcd}
	{\frac{1}{2}} & 0 & 1\ar[loop, out=-30, in=30, distance=20] && \infty\ar[loop, out=-30, in=30, distance=20, "2:1"']
	\arrow["{2:1}", from=1-1, to=1-2]
	\arrow[from=1-2, to=1-3]
\end{tikzcd}\]
We see $P_g = \{0,1,\infty\}$. There is an additional non-postcritical fixed point at $a=1/4$, so for $A =\{0,1,\infty,a\}$ the map rational Thurston map $(g,A)$ satisfies the main theorem.

Now consider the unique M\"{o}bius transformation $M_{a,0}$ which permutes the set $A$ as $(a\,\,0)(1\,\,\infty)$. In coordinates, this is given by
\[M_{a,0} = \frac{z-a}{z-1}.\]

The twisted map $f = M_{a,0}\circ g$ (for $a=1/4$) is a rational map given by
\[f(z) = \frac{16z^2-16z+3}{4(z-1)}.\]
It has dynamical portrait

\[\begin{tikzcd}
	{\frac{1}{2}} & \frac{1}{4} & 0 & \infty & 1.
	\arrow["{2:1}", from=1-1, to=1-2]
	\arrow[from=1-2, to=1-3]
	\arrow[from=1-3, to=1-4]
	\arrow[bend left,"{2:1}", from=1-4, to=1-5]
	\arrow[bend left, from=1-5, to=1-4]
\end{tikzcd}\]
Again, all points of the marking are postcritical points for $f$, so $(f,P_f)$ has a finite global curve attractor.
\end{exmp}

The author is not aware of a convenient sufficient criterion that determines whether a map is obtained by twisting a $(g,A)$ satisfying the main theorem by a flexible Latt\`{e}s map. There is, however, such a criterion for (post) twists by M\"{o}bius transformations. The result is due to Kelsey and Lodge \cite{kelseylodgequadratic}. Before describing their condition we must introduce some new terminology.

\begin{defn}
A marked rational map $(f,A)$ with $|A|=4$ is {\em statically reducible} if its static portrait has a component consisting of a single directed edge between points of $A$ with degree labeling $1$. Equivalently, $f$ is statically reducible if for some $a\in A$ we have
\[f^{-1}(f(a))\cap C_f=\emptyset\quad\text{and}\quad f^{-1}(f(a))\cap A=\{a\}.\]
If the above condition holds we call the point $a$ {\em statically trivial}.
\end{defn}

\begin{exmp}
Consider the map $(f,P_f)$ from the previous example \ref{twistexmp}, which we recall was given by twisting a map with fewer postcritical points by a M\"{o}bius transformation. The static portrait for the map $(f,P_f)$ is shown below:

\[\begin{tikzcd}
	0 & 1 & \infty & {\frac{1}{2}} & {\frac{1}{4}} \\
	\\
	\infty & 1 & {\frac{1}{4}} & 0
	\arrow[from=1-1, to=3-1]
	\arrow["2"', from=1-4, to=3-3]
	\arrow[from=1-5, to=3-4]
	\arrow["2"', from=1-3, to=3-2]
	\arrow[from=1-2, to=3-1]
\end{tikzcd}\]
The component $1/4\to 0$ in the graph shows that $f$ is statically reducible and $a=1/4$ is a statically trivial postcritical point.
\end{exmp}

The situation above is not a coincidence. It turns out static reducibility exactly detects when a map is obtained by M\"{o}bius twisting a map with fewer postcritical points. This leads us to the statement and proof of Kelsey and Lodge's criterion, which is \cite[Proposition 2.5]{kelseylodgequadratic}.

\begin{prop}\label{reduciblecriterion}
Suppose $(f,A)$ is a statically reducible marked rational map with $|A|=4$. Then there is a M\"{o}bius transformation $M$ with $M(A)=A$ and a rational map $(g,A)$ with $P_g \subset \{0,1,\infty\}$ such that $f=M\circ g$.
\end{prop}
\begin{proof}
Suppose $(f,A)$ has a statically trivial point in the marking, which we may assume is $a$. If $f(a)=a$ there is nothing to do, since we may put $M=\id$ and $g=f$. The point $a$ cannot be postcritical because no critical points map to $f(a)=a$, nor do any points of $A\setminus\{a\}$.

If $f(a)\neq a$, then there is a permutation which transposes these two points of the marking and also transposes the two points of $A\setminus\{f(a),a\}$. There is a unique M\"{o}bius transformation $M$ which induces this permutation on $A$, and we may put $g:=M\circ f$. The static portrait of this map is obtained from the static portrait of $f$ by applying the permutation $M$ to the target nodes and shifting the edges accordingly. In particular, $g(a)=a$ and $a$ is not a postcritical point of $g$, so $P_g\subset A\setminus \{a\}$. Since $M^2$ is a M\"{o}bius transformation fixing 4 points it must be the identity, so $M=M^{-1}$ and we may write $f=M\circ g$. This completes the proof.
\end{proof}

We now recall and prove our other main result described at the beginning of the paper.

\resmainthmtwo* 

\begin{proof}
All the work has already been done. By Proposition \ref{reduciblecriterion} we may write $f=M\circ g$ where $g$ satisfies Theorem \ref{mainthm}, and then we conclude by applying Proposition \ref{compattractor}.
\end{proof}

We remark that by further pre- or post-composing a rational Thurston map satisfying the above theorem with flexible Latt\`{e}s maps one can generate examples which have finite global curve attractors by construction, but which are no longer detected by the statically trivial postcritical point condition.

\section{Further Discussion}\label{futuresection}

In this section we discuss how the results of this paper may be extended even further. To begin, we review some known facts about the general correspondence on moduli space.

For a general marked Thurston map $(f,A)$ with $|A|\geq 3$, Koch \cite{SarahCritEndo} proved that $\sigma_f$ covers a finite algebraic correspondence\footnote{In fact, Koch proved the correspondence in the more general nondynamical setting of admissible covers $f\from(S^2,A)\to(S^2,B)$ where $A$ and $B$ are potentially different, but we shall restrict our attention to the dynamical case.}. In particular, we have the following commutative diagram:

\[\begin{tikzcd}
	{\calT_A} && {\calT_A} \\
	& {\calW_f} \\
	{\calM_A} && {\calM_A.}
	\arrow["{\sigma_f}"', from=1-3, to=1-1]
	\arrow["{\pi_A}"', from=1-1, to=3-1]
	\arrow["{\pi_A}", from=1-3, to=3-3]
	\arrow["{\omega_f}"', from=1-3, to=2-2]
	\arrow["Y", from=2-2, to=3-3]
	\arrow["X"', from=2-2, to=3-1]
\end{tikzcd}\]
The map $Y$ is a finite holomorphic covering map, $X$ is holomorphic, $\calW_f = \calT_A/H_f$ is a complex manifold of dimension $|A|-3$ obtained as the quotient of Teichm\"{u}ller space by the group of pure modular liftables $H_f$, and $\omega_f\from\calT_A\to \calW_f$ is the natural quotient map. This diagram also extends to the Weil-Petersson (WP) completion of Teichm\"{u}ller space in the obvious way. See \cite{selthesis} and \cite{KPSlifts} for more details regarding this point.

In general, the map $X$ is not always injective, so there is not always a map on moduli space going in the direction opposite $\sigma_f$ (as occurred in our setting in this paper). We write $X\circ Y^{-1}\from\calM_A\dto\calM_A$ to denote the multivalued map on moduli space going in the same direction as $\sigma_f$.

Before stating the theorem of Koch, Pilgrim, and Selinger, we must give a few more definitions. A nonempty subset $K\subset\calM_A$ is {\em invariant} under the correspondence $X\circ Y^{-1}$ if $X\circ Y^{-1}(K)\subset K$.

A length metric $l$ on $\calM_A$ is called {\em WP-like} if its lift $\tl$ on $\calT_A$ has the property that the identity map defines a homeomorphism between the completions of $\calT_A$ with respect to $\tl$ and with respect to the WP metric. The pullback correspondence $X\circ Y^{-1}$ on $\calM_A$ is {\em uniformly contracting} with respect to a WP-like metric $l$ on $\calM_A$ if there is a constant $0\leq \alpha<1$ such that for each curve $\gamma\from [0,1]\to\calM_A$ with finite length, and each lift $\tgamma$ of $\gamma$ under $Y$, we have $l(X\circ \tgamma)\leq \alpha l(\gamma)$.

\begin{thm}[\protect{\cite[Theorem 7.2]{KPSlifts}}]
Suppose $(f,A)$ is a rational Thurston map with hyperbolic orbifold.
\begin{enumerate}[leftmargin=2em, label={\rm (\roman*)}]
\item If the correspondence $X\circ Y^{-1}\from\calM_A\dto\calM_A$ has a nonempty compact invariant subset, then the virtual endomorphism $\phi_f$ is contracting, and $(f,A)$ has a finite global curve attractor.
\item If the correspondence $X\circ Y^{-1}\from \calM_A\to\calM_A$ is uniformly contracting with respect to a WP-like length metric $l$, then under iterated pullback every curve eventually becomes peripheral.
\end{enumerate}
\end{thm}

It is not hard to see how to obtain the extremal cases of Theorem \ref{mainthm} from here. Suppose $(f,A)$ satisfying the hypotheses of our main theorem is hyperbolic, meaning $B = P_\calF = \{0,1,\infty\}$. Since the ends $B$ of moduli space $\calM_A = \hC\setminus B$ all lie in superattracting cycles, deleting suitably chosen punctured disks centered at each end yields a compact set that is invariant under any branch of $f^{-1} = X\circ Y^{-1}$, and so part (i) applies. On the other hand, if $(f,A)$ is expanding, then $B = P_\calJ = \{0,1,\infty\}$, and the orbifold metric constructed in Section \ref{orbifoldsubsec} is a uniformly contracting WP-like metric and part (ii) applies.

Interestingly, the techniques used in the proof of the above theorem differ for parts (i) and (ii); part (i) uses the virtual endomorphism while part (ii) uses a specially designed metric lifted to Teichm\"{u}ller space. For us, using a common mechanism was crucial for the mixed cases (where some postcritical points were in the Fatou set and some were in the Julia set), since it allowed us to deal with attracting and repelling dynamics simultaneously. Perhaps translating the virtual endomorphism argument into a leashing argument similar to the one used in this paper will help push the theorem further.

There are other problems with this idea. One of the nice properties of our setting is that the ends of moduli space are well-sorted in the sense that repelling ends match with repelling ends under the correspondence, and likewise attracting ends match with attracting ends (this is another way of interpreting Lemma \ref{cuspsorting}). This sorting does not seem to occur in the general case, but perhaps it does for a sufficiently high iterate of the Thurston map. More investigation is required.
\bibliography{references}

\end{document}